\documentclass[12pt]{amsart}
\newcounter{defcounter}
\usepackage{latexsym, amsmath, amssymb, amsthm, mathrsfs,wasysym}
\usepackage[alwaysadjust]{paralist}
\usepackage{caption}
\usepackage{subcaption}
\usepackage[T1]{fontenc}
\usepackage{lmodern}
\usepackage[backref,colorlinks=true,linkcolor=blue,urlcolor=blue, citecolor=blue]%
  {hyperref}
\usepackage{amsrefs}
\usepackage[all]{xy}
\usepackage[T1]{fontenc}
\usepackage{mathptmx}
\usepackage{microtype}
\usepackage{framed}
\usepackage{tikz}
\usepackage{tikz-cd}
\usepackage{graphicx}
\usepackage[alwaysadjust]{paralist}
\makeatletter
\providecommand\@enum@widestlabel{7}
\makeatother

\newtheorem{lemma}{Lemma}[section]
\newtheorem{theorem}[lemma]{Theorem}
\newtheorem{corollary}[lemma]{Corollary}
\newtheorem{proposition}[lemma]{Proposition}
\newtheorem{conjecture}[lemma]{Conjecture}
\theoremstyle{definition}

\newtheorem{remark}[lemma]{Remark}
\newtheorem{example}[lemma]{Example}

\renewcommand{\theequation}%
{\arabic{section}.\arabic{lemma}.\arabic{equation}}
\newcommand{\CC}{\ensuremath{\mathbb{C}}} 
\newcommand{\NN}{\ensuremath{\mathbb{N}}} 
\newcommand{\PP}{\ensuremath{\mathbb{P}}} 
\newcommand{\QQ}{\ensuremath{\mathbb{Q}}} 
\newcommand{\RR}{\ensuremath{\mathbb{R}}} 
\newcommand{\ZZ}{\ensuremath{\mathbb{Z}}} 
\newcommand{\sI}{\ensuremath{\kern -1pt \mathscr{I}\kern -2pt}} 
\newcommand{\sJ}{\ensuremath{\kern -2pt \mathscr{J}\kern -2pt}} 

\newcommand{\sO}{\ensuremath{\mathscr{O}}}

\newcommand{\m}{\ensuremath{\mathfrak{m}}}

\renewcommand{\geq}{\geqslant}
\renewcommand{\leq}{\leqslant}

\DeclareMathOperator{\mult}{mult}

\DeclareMathOperator{\Sym}{Sym}

\newcommand{\equ}{\ensuremath{\,=\,}}
\newcommand{\deq}{\ensuremath{\stackrel{\textrm{def}}{=}}}

\newcommand{\Bplus}{\ensuremath{\textbf{\textup{B}}_{+} }}
\newcommand{\Bminus}{\ensuremath{\textbf{\textup{B}}_{-} }}
\newcommand{\Bstable}{\ensuremath{\textbf{\textup{B}} }}

\newcommand{\asy}[2]{\ensuremath{\sJ({#1},\| {#2} \|)}}

\begin{document}

\title{Singular divisors and syzygies of polarized abelian threefolds}

\author{Victor Lozovanu}

\address{Victor Lozovanu -- Current Address: Universit\'a degli Studi di Genova, 
\newline  
\hspace*{2.74in} Dipartimento di Matematica, 
\newline  
\hspace*{2.74in} Via Dodecaneso 35, 16146, Genova, Italy.\newline
\hspace*{2.74in} \textit{Email address}: \href{lozovanu@dima.unige.it}{\nolinkurl{lozovanu@dima.unige.it}}}

\address{\hspace*{1.9in} Address: Leibniz Universit\"{a}t Hannover,  \newline
 \hspace*{2.74in} Institut f\"{u}r Algebraische Geometrie, 
 \newline
 \hspace*{2.74in}
 Welfengarten 1, 30167, Hannover, Germany.
 \newline
\hspace*{2.74in} \textit{Email address}: \href{lozovanu@math.uni-hannover.de}{\nolinkurl{lozovanu@math.uni-hannover.de}}}

\maketitle

\begin{abstract}
We provide numerical conditions for a polarized abelian threefold $(A,L)$ to have simple syzygies, in terms of property $(N_p)$ and the vanishing of Koszul cohomology groups $K_{p,1}$. We rely on a reduction method of Lazarsfeld-Pareschi-Popa \cite{LPP11}, convex geometry of Newton--Okounkov bodies, inversion of adjunction techniques from work on Fujita's conjecture, and the use of differentiation by Ein-Lazarsfeld-Nakamaye \cite{ELN94}. As a by-product, we construct effective divisors in any ample rational class of high self-intersection, whose singularities are all concentrated on an abelian subvariety. This can be seen as the dual picture considered by Ein-Lazarsfeld \cite{EL97} for theta divisors.
\end{abstract}

\section*{Introduction}

The deep connection between the geometric data and the defining equations (syzygies) of a subvariety of  projective space has been known since classical times. But, aside from sporadic efforts, it was not until the work of M. Green \cite{G84} that a coherent general picture started to emerge. 

Green reinterpreted syzygies in terms of Koszul cohomology groups  $K_{p,q}(X,L, B)$ ($p,q\in\NN$), where $X$ is a projective variety, $L$ a very ample line bundle and $B$ an arbitrary line bundle on $X$. Out of Green's work two syzygetic properties have grown in prominence due to their connection to the geometry of $(X,L,B)$. First, property $(N_p)$ of Green--Lazarsfeld \cite{GL86} asks the ring $R(X,L)=\oplus_{m\in\NN} H^0(X, L^{\otimes m})$ to have linear syzygies up to step $p$ as a module over $\Sym^* H^0(X,L)$. Second one focuses on the vanishing of $K_{p,1}(X,L;dL)$ for  $d\gg 0$ (a nice overview is the recent survey of Ein--Lazarzfeld \cite{EL18} or the standard reference \cite{PAG}*{Section~1.8}). 

Both of these properties are well understood for curves. For property $(N_p)$ some milestones are the results of Green \cite{G84}, Green--Lazarsfeld \cite{GL86}, Farkas-Kemeny \cite{FK16}. For the vanishing of $K_{p,1}$ the main results are those  of Voisin \cite{V02}, and more recently  Ein--Lazarsfeld \cite{EL15}, and Aprodu-Farkas-Papadima-Raicu-Weyman \cite{AFPRW19}. On curves these two properties are closely related and are translated into very geometric properties that are in essence of Riemann-Roch type.

In higher dimensions the picture is murky. It's not clear what are the geometric properties that might have an impact on the syzygy side. For vanishing of Koszul groups $K_{p,1}(X,L,dL)$ recent work of Ein--Lazarsfeld-Yang \cite{ELY16} and Agostini \cite{A17} suggest some conjectural picture. 

For property $(N_p)$ abelian varieties have proven to be a successful case study.
Extrapolating elliptic curves, Lazarsfeld conjectured that $L^{\otimes (p+3)}$ satisfies property $(N_p)$ for an ample line bundle $L$ on a complex abelian variety $A$. This was proved by Pareschi in \cite{P00} (see also \cite{PP03}). It is worth pointing out that $L^{\otimes 2}$ is base point free, i.e.  $p=-1$. So, property $(N_p)$ ($p\geq 0$) can be perceived syzygetically and positivity-wise, as a natural generalization of global generatedness ($N_{-1}$).  

One drawback is that Preschi's result feels incomplete. It doesn't deal with the actual line bundle $L$ and doesn't take into account the numerical data of $(A,L)$. More importantly, it doesn't connect the geometric data to the syzygetic one, by showcasing those subvarieties with the strongest impact on the syzygies of the ambient space. To amend this we propose the following conjecture:
\begin{conjecture}\label{conj:main1}
	Let $(A,L)$ be a $g$-dimensional complex polarized abelian manifold and $p\geq -1$ an integer. Suppose that for any abelian submanifold $A'\subseteq A$ we have
	\[
	\Big(A'\cdot L^{\textup{dim}(A')}\Big)  \ > \ \Big((p+2)\cdot \textup{dim}(A')\Big)^{\textup{dim}(A')} \ . 
	\]
Then $(A,L)$ satisfies property $(N_p)$. Moreover, if $p\geq 0$, then $K_{p,1}(X,L;dL)= 0$ for any $d\geq 2$.
\end{conjecture}

Originally Ito \cite{I18} proposed this conjecture for property $(N_p)$ ($p\geq 0$). We include here base-point freeness as property $(N_{-1})$. And, as for curves, we expect the vanishing of $K_{p,1}$ on abelian manifolds to behave similarly as property $(N_p)$.

In dimension two Conjecture \ref{conj:main1} for property $(N_p)$ was proven by Ito \cite{I18}, by streamlining the argument of a similar statement in \cite{KL18}, by K\"uronya and the author. Moreover, using the work of Aprodu--Farkas \cite{AF11} and Voisin \cite{V02}  on Green's conjecture for curves, Agostini, K\"uronya and the author \cite{AKL19} obtain similar results for K3 and Enriques surfaces. Interestingly, a unified proof for all surfaces with numerically trivial canonical bundle is still missing. 

The main goal of this article is to give credence to Conjecture \ref{conj:main1} in dimension three. More precisely, we show the following theorem:
\begin{theorem}\label{thm:main}
Let $(A,L)$ be a complex polarized abelian threefold such that $(L^3)>59\cdot (p+2)^3$ for some integer $p\geq 0$. Assume  the following conditions:
\begin{enumerate}
\item $(L^2\cdot S)> 4\cdot (p+2)^2$ for any abelian surface $S\subseteq A$. 
\item $(L\cdot C) \ >  \ 2\cdot (p+2)$ for any elliptic curve $C\subseteq A$.
\end{enumerate}
Then the pair $(A,L)$ satisfies property $(N_p)$ and $K_{p,1}(A,L;dL)=0$ for any $d\geq 2$. Moreover, Conjecture \ref{conj:main1} holds for $p=-1$, i.e. for globally-generatedness.
\end{theorem} 
As a consequence of their work on Bridgeland stability conditions  on abelian threefolds Bayer--Macr\'i--Stellari \cite[Corollary 1.5]{BMS16} provide numerical results for base-point freeness and very ampleness (weaker than $N_0$) in this setup. Our main results are obtained with different methods and for a larger class of properties. We also provide better numerical data and exhibit exactly those subvarieties impacting the most the syzygies.

Based on \cite{LPP11} and \cite{ELY16}, we translate our global syzygetic properties into finding effective rational divisors from $L$ with certain singularities. So, as a by-product this brings to light a conjectural picture on the behaviour of singularity loci for effective divisors on abelian manifolds. 
\begin{conjecture}\label{conj:main2}
	Let $A$ be a $g$-dimensional complex abelian manifold and $B$ an ample $\RR$-divisor on $A$ with $(B^g) > g^g$. Then there exists an effective divisor $D\equiv B$, whose singularities are all concentrated on an abelian subvariety of $A$. More concretely, the log-canonical center of $D$ is a smaller-dimensional abelian subvariety of $A$. 
\end{conjecture}
To put this into perspective, the study of singularities for pluri-theta divisors, when $B=\Theta$, has attracted lots of attention. For example, Ein-Lazarsfeld \cite{EL97} show that any effective divisor $D\equiv \Theta$ has trivial log-canonical center, except when $(A,\Theta)$ splits as a product of elliptic curves. Conjecture \ref{conj:main2} turns the picture upside down and deals with the opposite case $B^g\geq g^g\gg \Theta^g=g!$. 

Under our conditions, there is always divisors with non-trivial log-canonical center, but the hard part is to force these centers to be abelian. This type of statements opens the door to induction type arguments when studying the properties of the pair $(A,B)$.
\begin{corollary}\label{cor:main}
If $A$ is an abelian threefold and $B$ an ample $\RR$-divisor on $A$ with $(B^3)>40$, then there is an effective divisor $D\equiv B$, whose log-canonical center is either the origin, an elliptic curve or an abelian surface.
\end{corollary}

\subsection*{Strategy of the proof}
Motivated by Hwang--To \cite{HT11}, \cite{LPP11} and some simplification from \cite{I18} reduce the verification of $(N_p)$ to constructing singular divisors with zero-dimensional canonical center. Aided by Ein--Lazarsfeld--Yang \cite{ELY16}, we first realize that the same reduction is possible for the vanishing of the Koszul cohomology groups $K_{p,1}$.

With this in hand, we turn our attention to the blow-up  $\pi: \overline{A}\rightarrow A$ at the origin $0\in A$ with  exceptional divisors $E\simeq\PP^2$. So, to find our singular divisors, we study the function
\[
 t\ \in \RR_+ \ \longrightarrow \ \Bstable\big(\pi^*L-tE\big) \ \subseteq \overline{A}\ . 
\]
Our goal is to study the variation, as a function of "$t$", of the stable base loci $\Bstable\big(\pi^*L-tE\big)$  and the behaviour of the asymptotic multiplicities, as developed in \cite{ELMNP06}, along its irreducible components.

Inspired by work of Faltings in diophantine geometry, Ein--Lazarsfeld--Nakamaye \cite{ELN94} pioneered the idea of differentiation in algebraic geometry. Due to \cite{N96}, on abelian varieties the multiplicity function of any irreducible component of $\Bstable(\pi^*(L)-tE)$ has large slope, as "$t$" varies. So, if this locus contains a non-degenerate subvariety, then $\Bstable(\pi^*(L)-tE)$ start containing pretty quickly sums of this subvariety until finally $\pi^*(L)-tE$ is no more pseudo-effective. 

As in \cite{KL19}, Newton--Okounkov bodies enable us to quantify numerically this geometric data, as they encode intersection numbers and form a powerful computational tool. So, using the infinitesimal picture of positivity, developed in  \cite{KL18,KL17}, the behaviour of our base loci prescribe restrictions to the shape of infinitesimal Newton-Okounkov bodies as convexs set in $\RR^3$. As their Euclidean volume is $L^3/6$, so quite large, we get a contradiction.

When our base loci consists only of abelian subvarieties we apply inversion of adjunction type techniques, as in \cite{EL93, AS95, Ka97, H97}. For example, when the intersection numbers of $L$ with respect to these loci are quite large, then by changing divisors we can cut the log-canonical center and be in good shape. Otherwise, the Seshadri contant of $L$ is forced to be small, which with some more work leads us to a contradiction of the conditions we have in the main statement.

\subsection*{Acknowledgements} The author thanks Alex K\"uronya for the mathematical support during this project. He is grateful to Lawrence Ein, Victor Gonzalez, Klaus Hulek and Mike Roth for helpful discussions, Atsushi Ito for providing an early copy of \cite{I18}, and Rikito Ohta for spotting a mistake in an earlier version. The connection between the vanishing of $K_{p,1}$ and singular divisors was previously discussed with Daniele Agostini and A. K\"uronya, allowing the author to publish it here.

\section{Preliminary results}

We start by presenting some of the preliminary techniques we use later. Most parts might be known to experts, but for the benefit of the reader we try to be as concise and complete as possible. Most notation and definitions in this article follow closely the standard textbooks \cite{Har} and \cite{PAG}.

\subsection{Asymptotic multiplicity, multiplier ideals and base loci}
In \cite{PAG} the theory of asymptotic multiplier ideals for line bundles is presented. For our purposes we need the more general one for big $\QQ$-divisors, which is not much different. So, in the following we state their main properties and show how they can be deduced from the same theory for line bundles.

	Let $X$ be a smooth complex projective variety, $D$ an effective big $\QQ$-Cartier divisor on $X$. Let $m>0$ an integer such that $mD$ is an integral Cartier divisor. The \textit{asymptotic multiplier ideal} of $D$ 
	\[
	\asy{X}{D} \deq \sJ(X, \frac{1}{m}\cdot \| mD \|  )\ ,
	\]
	where the right-hand side is the appropriate asymptotic multiplier ideal of the line bundle $\sO_X(mD)$.

\begin{lemma}\label{lem:multiplier}
	Let $X$ be a smooth projective variety and $D$ an effective big $\QQ$-Cartier divisor on $X$. 
	\begin{enumerate}
		\item The ideal sheaf $\asy{X}{D}$ is a numerical invariant and doesn't depend on the choice of $m$.
		\item For any $D'\equiv D$ effective $\QQ$-divisor, we have the inclusion of ideals
		\[
		\sJ(X,D') \ \subseteq \ \sJ(X,||D||)\ .
		\]
		Let  $D''\in |mD|$ be a general element for some $m\in \NN$  sufficiently large and divisible. Then for $D'=\frac{1}{m}D''$ the inclusion turns out to be an equality.
	\end{enumerate}
	
\end{lemma}
\begin{proof} $(1)$ For those $m_1,m_2>0$ that make sense, \cite[Theorem 11.1.8]{PAG} yields
	\[
	\sJ(X,\frac{1}{m_1}\cdot \| m_1D  \|  ) \equ \sJ(X,\frac{1}{m_1m_2}\cdot \| m_2m_1D  \|  )
	\equ \sJ(X,\frac{1}{m_2}\cdot \| m_2D  \|  )\ .
	\]
	Together wit \cite[Example 11.2.12]{PAG}, these imply the statement. 
	
	$(2)$ Taking $m\gg 0$ and divisible enough, we have the following inclusions
	\[
	\sJ(X,D' ) \subseteq  \sJ(X, \frac{1}{m}\cdot | mD'| ) \equ \sJ(X;\frac{1}{m}\cdot || mD' || ) \ = \ \sJ(X;\frac{1}{m}\cdot || mD || )\ = \ \sJ(X, || D || ) \ .
	\]
	The first inclusion is implied by \cite[Proposition 9.2.32]{PAG}. The equalities follow in order from \cite[Proposition 11.1.4 and Example 11.3.12]{PAG} and the definition. The inclusion becomes an equality exactly as in our statement is due to \cite[Proposition 9.2.26]{PAG}.
\end{proof}

\subsection{Base loci of numerical classes}
Let $D$ be a big $\QQ$-divisor on a complex smooth projective variety $X$. The \textit{stable base locus} of $D$ is the Zariski-closed set 
\[
\Bstable (D) \ \deq \ \bigcap_{m>0} \textup{Bs}(mD) \ ,
\]
where the intersection is taken over all $m\gg 0$ with $mD$ being a Cartier divisor. The set $\textup{Bs}(mD)$ is the base locus of the linear system $|mD|$. It turns out that $\Bstable(D)$ is not a numerical invariant.

For this reason \cite{ELMNP06} introduces two approximations of this locus. The \textit{augmented base locus} and the \textit{restricted base locus} of $D$ are defined as follows
\[
\Bplus(D) \ \deq \ \bigcap_{m>0}\Bstable(D-\frac{1}{m}A) \ \textup{ and (resp.) } \ \Bminus(D) \ \deq \ \bigcup_{m>0}\Bstable(D+\frac{1}{m}A)\ ,
\]
where $A$ is an ample class. Note $\Bminus(D)\subseteq \Bstable(D)\subseteq \Bplus(D)$. More importantly, these two new loci don't depend on the choice of the ample class $A$ and are numerical invariants of $D$. However, $\Bminus(D)$ might not be Zariski closed, see \cite{Le14}, but $\Bplus(D)$ is always closed. Luckily, for us these loci can be assumed to be equal. So, any strange phenomenon does not interfere with our computations.

To measure how badly the numerical class of $D$ vanishes along a subvariety $V\subseteq X$ we use the multiplicity. So, the \textit{asymptotic multiplicity} of $D$ along $V$ is defined to be
\[
\mult_V(||D||) \ = \ \lim_{m\rightarrow \infty}\ \frac{\mult_V(|mD|)}{m} \ .
\]
By \cite{ELMNP06}*{Proposition~2.8}, it turns out that $\mult_V(||D||)>0$ if and only if  $V\subseteq \Bminus(D)$.

Through out the paper we use frequently the following lemma in an implicit way.
\begin{lemma}\label{lem:multiplicity}
	Let $\epsilon>0$ be some real number. Then the following two conditions are equivalent:
	\begin{enumerate}
		\item $\mult_V(||D||) \ \geq  \ \epsilon$.
		\item $\mult_V(D') \ \geq  m\epsilon$ for all $m\gg 0$ and divisible enough and any generic choice of $D'\in|mD|$.
	\end{enumerate}
\end{lemma}
\begin{proof}
	By \cite[Lemma 3.3]{ELMNP06}, we have the following equality
	\[
	\mult_V(||D||) \ = \ \underset{D'\equiv D, D\geq 0}{\textup{inf}}\ \mult_V(D') \ .
	\]
Moreover, $\mult_V(|mD|)=\mult_V(D')$ for a generic $D'\in |mD|$. Together with the definition of asymptotic multiplicity, this implies easily the statement.
\end{proof}

\subsection{The geometry of the blow-up}
Let $X$ be a smooth projective variety and $\pi:\overline{X}\rightarrow X$ the blow-up of $X$ at a point $x\in X$ with exceptional divisor $E$. For an ample $\QQ$-divisor $B$ on $X$ denote by
\[
B_t \ \deq \ \pi^*(B) \ - \ tE, \textup{ for any } t\geq 0 \ .
\]
To this data we can associate two invariants. First, \textit{the Seshadri constant} of $B$ at $x$ is defined to be 
\[
\epsilon(B,x) \ \deq \ \inf_{x\in C}\frac{(B\cdot C)}{\mult_x(C)} \ = \ \sup\{t>0 \ | \ B_t \textup{ is nef (ample) }\} \ > \ 0 \ .
\]
The second invariant, called \textit{the infinitesimal width} of $B$ at $x$, is given by the formula
\[
\mu(B,x) \ \deq \ \sup\{t>0 \ | \ B_t \textup{ is }\QQ-\textup{effective}\} \ < \ \infty\ .
\]
Notice that $\mu(B,x)\geq \epsilon(B,x)$ (see \cite[Chapter V]{PAG} for more information on the Seshadri constant). \footnote{These invariants can be defined similarly for singular spaces, and in Section~3.2 we do so for surfaces.} 
\begin{lemma}\label{lem:baselocus}{(\textbf{Behavior of infinitesimal base loci}).}
	\begin{itemize}
		\item[(a)] The function $t\rightarrow \ \Bminus(B_t)$ or $\Bplus(B_t)$ is increasing with respect to inclusion.
		\item[(b)] Let $t_0\geq 0$ be some positive real number. Then there is $0<\delta \ll 1$ such that 
		\[
		\Bminus(B_t)\ =\ \Bstable(B_t)\ =\ \Bplus(B_t)
		\]
		is constant for any $t\in (t_0,t_0+\delta)$.
	\end{itemize}
\end{lemma}
\begin{remark}
	Consequently, there  are countably many points where these base loci can jump. We don't know if the set of jumps has (in)finitely many accumulation points, or just $\mu(B,x)$. In \cite{KLM12} it is shown that for surfaces there are finitely many jumps.
\end{remark}
\begin{proof}
	$(a)$ As $E$ is effective, the definition of the stable base locus implies that the function $t\rightarrow \Bstable(B_t)$ is increasing. Again by definition, the same can be said about the other two loci.
	
	$(b)$ Without loss of generality assume $t_0\geq \epsilon(L,x)$. Let $r\in (0, \epsilon(B,x))$ be some rational number. For any $0<\lambda\ll 1$ we have have the following equalities of sets
	\[
	\Bplus(B_{t_0}) \ = \ \Bstable\Big(B_{t_0}-\lambda B_r\Big)\ =\ \Bminus\Big(B_{t_0}-\lambda B_r\Big) \ = \ \Bstable\Big(B_{\frac{t_0-\lambda \cdot r}{1-\lambda}}\Big) \ .
	\]
	As $B_r$ is ample, the first equality follows from the definition of augmented base locus. The function $t\rightarrow \Bstable(B_t)$ is increasing, so the definition of the restricted base yields the second one. Finally, our asymptotic base loci don't change by taking powers of the class, so we get the third one. 
	
Finally, as $t_0-\lambda \cdot r > (1-\lambda)t_0$, then $	\Bplus(B_{t_0})= \Bstable(B_{t_0+\delta}) =\Bminus(B_{t_0+\delta})$	for any $0<\delta\ll 1$.
\end{proof}
Fix $\overline{Y}\subseteq X'$ a subvariety. Let $m_{\overline{Y}} : [0,\mu(B,x)] \rightarrow \RR_+$ be \textit{the multiplicity function}, given by
	\[
m_{\overline{Y}}(t)\ \deq \ \textup{mult}_{\overline{Y}}\Big(||\pi^*(B)-tE||\Big) \ .
	\] 
It is also important to encode, when this function start being non-zero. So, we define
	\[
	t_{\overline{Y}}\ \deq \ \textup{min}\{t>0\ |\ \overline{Y}\subseteq \Bminus(B_t)\} \ .
	\]
	Without any sort of confusion, for a subvariety $Y\subseteq X$ we denote the above quantities by $t_Y$ and $m_Y$, as computed for the proper transform $\overline{Y}$ of $Y$ on the blow-up $X'$.
\begin{lemma}\label{lem:convexity}{(\textbf{Convexity}).}
$m_{\overline{Y}}: [0,\mu(B,x)] \rightarrow \RR_+$ is a continuous, increasing and convex function.
\end{lemma}
\begin{proof}
Since $E$ is effective, we can see easily that $m_{\overline{Y}}$ is increasing. To prove convexity and continuity, let $0<t_0\leq t_1\leq \mu(B,x)$ and $s\in (0,1)$. Then we have the following inequalities
	\[
	(1-s)m_{\overline{Y}}(t_0)  \ + \ s\cdot m_{\overline{Y}}(t_1) \ \geq \ \textup{mult}_{\overline{Y}}\Big(||(1-s)B_{t_0}+sB_{t_1}||\Big) \ = \ m_{\overline{Y}}((1-s)t_0+st_1) \ .
	\]
	The inequality is due to rescaling and convexity of the function $\xi\in\textup{N}^1(X')_{\RR}\rightarrow \mult_{\overline{Y}}(||\xi||)\in \RR$ from \cite[Remark 2.3 and Proposition 2.4]{ELMNP06}. Ditto for continuity and we finish the proof.
\end{proof}

\subsection{Inversion of adjunction}
Inversion of adjunction is a powerful tool in algebraic geometry, providing ways to cut down the singular locus of divisors. The material is inspired by \cite{E97} and \cite{PAG}.

For an effective $\QQ$-divisor $D$ on a smooth projective variety $X$, the \textit{log-canonical threshold} is
\[
\textup{lct}(D) \ \deq \ \textup{inf}\{c>0 \ | \ \sJ(X;c\cdot D) \ \neq \ \sO_X\} \ .
\]
The divisor $D$ is called \textit{log-canonical} if $\textup{lct}(D)=1$. The \textit{log-canonical locus} of $D$  is defined to be 
\[
\textup{LC}(D) \ \deq \ \textup{Zeroes}\big(\sJ(X;cD)\big) \ , \textup{ where } c=\textup{lct}(D) \ .
\]
Modifying the divisor we can get the log-canonical center to be irreducible.
\begin{lemma}\label{lem:trick}{(\textbf{Tie-breaking trick}).}
	Let $D$ be an effective $\QQ$-divisor on $X$ and $c=\textup{lct}(D)$. There exists an effective divisor $E$ in some ample class $B$, satisfying the property that for any $0<\epsilon\ll 1$ there are rational numbers $0<c_{\epsilon}<c$ and $t_{\epsilon}>0$ such that the lc-locus of the divisor 
	\[
	D_{\epsilon} \ \deq \ c_{\epsilon}D\ + \ t_{\epsilon} E \ 
	\]
	is an irreducible normal subvariety of $\textup{LC}(D)$. Moreover, $c_{\epsilon}\rightarrow c$ and $t_{\epsilon}\rightarrow 0$ whenever $\epsilon\rightarrow 0$.
\end{lemma}
\begin{proof}
This statement is a global analogue of \cite[Lemma 10.4.8]{PAG} and the proof is exactly the same, so we don't reproduce it here. Normality instead follows from \cite[Theorem 1.6]{Ka97}. 
\end{proof}
We say $Z\subseteq X$ is a \textit{critical subvariety} of $D$ if $Z=\textup       {LC}(D_{\epsilon})$  as in Lemma~\ref{lem:trick}. The proof of \cite[Lemma 10.4.8]{PAG} shows $Z$ is independent of the choice of  $0<\epsilon\ll 1$.  
Moreover, \cite[Remark 2.1]{I18} shows some interesting phenomena when the critical variety is zero-dimensional.
\begin{lemma}\label{lem:ito}
Let $D\equiv B$ be an effective $\QQ$-divisor on $X$ with $\textup{lct}(D)<1$, where $B$ is an ample class. If $\textup{LC}(D)$ is zero-dimensional, then there is another effective $\QQ$-divisor $D'\equiv c'B$ with $0<c'<1$ satisfying the property that $\sJ(X,D')=\m_x$ for some $x\in X$.
\end{lemma}
\begin{proof}
We can assume that $D$ is log-canonical. Let $D'$ be as in Lemma~\ref{lem:trick} and $x\in X$ the unique point in the co-support of $\sJ(X,D')$. Let $\pi' : Y \longrightarrow X$	be a log-resolution of $D'$, obtained by blowing up smooth centers on $X$. Now, any irreducible exceptional divisor $F$ on $Y$ that counts in the construction of $ \sJ(X;D')$ is contracted to $x$ by $\pi'$ and has coefficient $-1$. In particular, $\pi'$ factors out through the blow-up $\pi:X'\rightarrow X$ at $x$ and exceptional divisor $E$. Furthermore, any such $F$ maps to $E$. So, taking any $f\in\m_x$, then $\mult_F((\pi')^*(f))\geq 1$ as we have automatically $\mult_E(\pi^*(f))\geq 1$. In particular, $\m_x\subseteq \sJ(X;D')$ and since the latter is not-trivial it is actually an equality.
\end{proof}
We explain inversion of adjunction in two cases, where one is definitely known to the experts.
\begin{proposition}[\textbf{Cutting down the LC locus I}]
	\label{prop:cuttingdown1}
	Let $D\equiv cB$ a log-canonical effective $\QQ$-divisor for some $0<c<1$, where $B$ is ample. If $V=\textup{LC}(D)$ is an irreducible smooth curve and
	\[
	(1-c)\cdot (B\cdot V) \ > \ 1 \ ,
	\]
	then there is an effective $\QQ$-divisor $D'\equiv (1-c)B$ such that the LC-locus of $(1-\delta)D+D'$ is zero-dimensional for any $0<\delta\ll 1$. 
\end{proposition}
\begin{proof}
	 By Riemann-Roch for any $0< \xi \ll 1$ there is an effective divisor $D^{\xi}_V\equiv B|_V$ with 
	\[
	\mult_x(D^{\xi}_V) \ \geq (B\cdot V) \ - \ \xi 
	\]
	at a fixed point $x\in V$ and $\mult_y(D^{\xi}_V)\leq \xi$ at any $y\in V\setminus \{x\}$.
	
Let $m\gg 0$ be a divisible enough interger, so that $mD^{\xi}_V$ is Cartier, $mB$ defines a very ample line bundle, the restriction map on global sections $H^0(X,\sO_X(mB))\rightarrow H^0(V,\sO_V(mB))$ is surjective and $\sI_{V|X}(mB)$ is globally generated. All these implies that the linear series of divisors 
	\[
	\{D^{\xi}_m\in|mB|, \textup{ where } D^{\xi}_m|_V=mD^{\xi}_V\}
	\]
	 has no base points on $X\setminus V$. Let $D^{\xi}\deq\frac{1-c}{m}D^{\xi}_m$ for some general $D^{\xi}_m\in|mB|$ with $D^{\xi}_m|_V=mD^{\xi}_V$. Then we have the equality of ideal sheaves
	\[
	\sJ(X,(1-\delta)D+D^{\xi})|_{X\setminus V} \ =  \ \sJ(X,(1-\delta)D)|_{X\setminus V} \ = \ \sO_{X\setminus V}\ ,
	\]
	for any $0<\delta\ll 1$, by applying \cite[Example 9.2.29]{PAG} and the fact that $D$ is log-canonical. The same equalities hold at $y\in V\setminus\{x\}$ and taking $0<\xi\ll\delta\ll 1$. 
	
	It remains to show that $\sJ(X,(1-\delta)D+D^{\xi})$ is non-trivial at $x$. As $V$ is a smooth curve, then $\mult_x(D^{\xi})>1$ whenever $\xi\ll 1$. The statement is then implied by \cite[Proposition 10.4.10]{PAG}.
\end{proof}

\begin{proposition}[\textbf{Cutting down the LC locus II}]
	\label{prop:cuttingdown2}
	Let $B$ be an ample class on $X$ and $D\equiv cB$ be a log-canonical effective $\QQ$-divisor for some $0<c<1$. Suppose for some point $x\in X$ we have
	\[
	\mult_x(D)+(1-c)\epsilon(B,x) > \textup{dim}(X) \ .
	\] 
	Then there is an effective divisor $D'\equiv (1-c)B$ with $\textup{LC}((1-\delta)D+D') = \{x\}$	for any $0<\delta \ll 1$.	
\end{proposition}
\begin{proof}
	The inequality in the statement is still valid if we tweak it a little bit, say as follows
	\[
	(1-\delta)\mult_x(D)+(1-c)(\epsilon(B,x)-\delta) > \textup{dim}(X) \ ,
	\]
	for some $0<\delta\ll 1$. Take $m\gg 1$ and divisible enough. Then the linear system
	\[
	\{D_m\in |mB| \ | \ \mult_x(D_m)\geq m(\epsilon(B,x)-\delta)\}
	\]
	is base point free on $X\setminus\{x\}$. Setting $D'=\frac{1-c}{m}D_m$, then we have 
	\[
	\sJ(X,(1-\delta)D+D')|_{X\setminus \{x\}} \ =  \ \sJ(X,(1-\delta)D)|_{X\setminus\{x\}} \ = \sO_{X\setminus\{x\}},
	\]
	for any $0<\delta\ll 1$, by \cite[Example 9.2.29]{PAG} and the fact that $D$ is log-canonical. So, the ideal $\sJ(X,(1-\delta)D+D')$ is non-trivial at most at $x$. But the first inequality in the proof yields
	\[
	\mult_x((1-\delta)D+D') > \textup{dim}(X), \textup{ for any } 0<\delta \ll 1 \ .
	\]
So, applying \cite[Proposition 9.3.2]{PAG} we get $\textup{LC}((1-\delta)D+D')=\{x\}$.	
\end{proof}

\section{Convex geometry and the infinitesimal picture of threefolds}

In this section we introduce infinitesimal Newton--Okounkov bodies on threefolds and study how base loci affect their shape. Finally, we discuss infinitesimal data on singular surfaces.

In the following $X$ is a complex projective smooth threefold and $\pi:X'\rightarrow X$ the blow-up of $X$ at a point $x\in X$ with the exceptional divisor $E\simeq \PP^2$. Let $\mu\deq \mu(B,x)$ be the infinitesimal width and $\epsilon\deq\epsilon(B,x)$ the Seshadri constant of an ample $\QQ$-class $B$ at $x$.
\subsection{Infinitesimal Newton--Okounkov bodies}
We refrain from a detailed exposition on Newton--Okounkov bodies. Instead, we refer to the original sources \cite{KK12, LM09}, or the survey articles \cite{B, KL_Survey}. Their connection to local positivity of line bundles was tackled in \cite{CHPW18, KL15,KL17,KL18,Ro16}. 

We start with a short overview of the infinitesimal picture. So, fix an \textit{infinitesimal flag} 
\[
Y_{\bullet} \ : \ Y_0=X' \ \supseteq \ Y_1=E \ \supseteq \ Y_2=l \ \supseteq \ Y_3=\{Q\} \ ,
\]
where $l\subseteq E\simeq \PP^2$ is a line and $Q\in l$ is a point. We say that the flag $Y_{\bullet}$ is \textit{generic} if $l\subseteq \PP^2$ is a generic line and $Q\in l$ a generic point.

The \textit{infinitesimal Newton--Okounkov body} of $B$ encodes how all the effective divisors in its class vanish along the tangency directions given by the flag. More specifically, it is defined as
\[
\tilde{\Delta}_{Y_{\bullet}}(B)\ = \ \textup{convex hull}\{\nu_{Y_{\bullet}}(D)\ | \ \textup{ effective } D\equiv B\} \ \subseteq \ \RR^3 \ .
\] 
We define $\nu_{Y_{\bullet}}(D)=(\nu_1,\nu_2,\nu_3)$ as follows. Set $\nu_1=\mult_x(D)$. Then $E$ is not contained in the support of $\pi^*(D)-\nu_1E$ and its restriction $D_1$ to $E$ makes sense as an effective divisor. Set finally $\nu_2=\mult_l(D_1)$ and $\nu_3= \mult_Q\big(\big(D_1-\nu_2l\big)|_l\big)$. 

In the following \textit{the inverted simplex} of length $\sigma$ is the convex set defined to be
\[
\Delta^{-1}_{\sigma}\deq\{(t,u,v)\in\RR^3_+\ | \ 0\leq t\leq \sigma, 0\leq u+v\leq t\} \ \subseteq \ \RR^3_+\ ,
\]
With this in hand we collect the most usefull properties of infinitesimal Newton--Okounkov bodies.
\begin{theorem}\label{thm:propertiesnobody}{(\textbf{Properties of infinitesimal Newton--Okounkov bodies}).}
		\begin{enumerate}
		\item The Euclidean volume $\textup{vol}_{\RR^3}\big(\tilde{\Delta}_{Y_{\bullet}}(B)\big)\ = \ B^3/6$.
		\item $\tilde{\Delta}_{Y_{\bullet}}(B) \subseteq \Delta^{-1}_{\mu}$ and $\tilde{\Delta}_{Y_{\bullet}}(B)\cap \{\mu\}\times \RR^2\neq\varnothing$, where $\mu= \mu(B,x)$.
		\item $\epsilon(B,x)=\textup{max}\{\sigma| \Delta^{-1}_{\sigma}\subseteq \tilde{\Delta}_{Y_{\bullet}}(B) \}$, so the right side doesn't depend on the choice of $Y_{\bullet}$. 
		\item If the flag $Y_{\bullet}$ is generic, then $[0,\mu]\times (0,0)\subseteq \tilde{\Delta}_{Y_{\bullet}}(B)$.
	\end{enumerate}
	
\end{theorem}
Statement $(1)$ is the main result of \cite{LM09}, $(2)$ and $(3)$ follow from \cite[Proposition 2.6 and Theorem C]{KL17}. Lemma~\ref{lem:baselocus} and \cite[Theorem 2.1]{KL15} yield the last statement.

Moving forward, Theorem~\ref{thm:propertiesnobody} fully describes $\tilde{\Delta}_{Y_{\bullet}}(B)$ in the region $[0,\epsilon]\times\RR^2$. Outside of it, big divisors with base loci appear and change the shape.
\begin{proposition}\label{prop:nobodies}{(\textbf{Conditions on infinitesimal Newton--Okounkov bodies}).}
	Let $x\in V \subseteq X$ be a subvariety with $m_V(t)=m_t>0$ for some $t\in (\epsilon, \mu)$. Fix $Y_{\bullet}$ a generic infinitesimal flag.
	\begin{enumerate}
		\item If $V$ is a curve, then
		\[
		\tilde{\Delta}_{Y_{\bullet}}(B) \ \bigcap \ \{t\}\times\RR^2 \ \subseteq \ \textup{convex hull}\{((t,0,0),(t,0,t),(t,t-m_t, m_t), (t,t-m_t,0)\} \ .
		\]
		\item If $V$ is a surface with $m=\mult_x(Y)\geq 1$, then 
		\[
		\tilde{\Delta}_{Y_{\bullet}}(B) \ \bigcap \ \{t\}\times\RR^2 \ \subseteq \ \textup{convex hull}\{((t,0,0),(t,0,t-m\cdot m_t),(t,t-m\cdot m_t, 0)\} \ .
		\]
	\end{enumerate}
\end{proposition}

\begin{proof}
Note first that the vertical slices $\tilde{\Delta}_{Y_{\bullet}}(B)\cap\{t\}\times\RR^2$ change continuously with $t$. So, we can assume that all the base loci of $B_t$ agree, by Lemma~\ref{lem:baselocus}. Moreover, due to genericity $l$ and $Q$ are not contained in these loci. Finally, denote by $\overline{V}$ the proper transform of $V$ through the blow-up.

	$(1)$ Fix a point $P\in \overline{V}\cap E\neq \varnothing$, which automatically is distinct from $Q$ and not contained in $l$. The definition of multiplicity yields $\mult_{P}(||B_t|| )\geq m_t$. Now, let  $D\equiv B$ be an effective $\QQ$-divisor with $\mult_x(D)=t$, i.e. its proper transform $\overline{D}\equiv B_t$.  Let $\nu_{Y_{\bullet}}(D) = (t,a,b)$ for some $a,b\in \QQ$. Lemma~\ref{lem:multiplicity} yields $\textup{mult}_P(\overline{D})\geq m_t$ and restricting to $E$ forces $\textup{mult}_P(\overline{D}|_E)\geq m_t$. As $P\notin l$, then $a\leq t-m_t$.
	
Once we have this inequality our statement follows from definitions. So, assume it doesn't hold, i.e. $a>t-m_t$. Then $R\deq \overline{D}|_E$ is an effective  $\QQ$-divisor of degree $t$ on $\PP^2$. By the definition of $\nu_{Y_{\bullet}}$, we can write $R = al + R'$, where $l\nsubseteq \textup{Supp}(R')$. Since $P\notin l$, the properties of $R$  yield 
	\[
	\mult_P(R') \ \geq \ m_t \textup{ and } \textup{deg}(R')\ = \ t-a \ < \ m_t \ .
	\]
We get our contradiction, as B\'ezout's theorem implies that there is no such $R'$ in $\PP^2$.	
	
	$(2)$ As before, let $D\equiv B$ be an effective $\QQ$-divisor with $\nu_{Y_{\bullet}}(D) = (t,a,b)$, with $a,b\in \QQ$. By Lemma~\ref{lem:multiplicity}, its proper transform can be written as $\overline{D}= D'+(m_t+x)\overline {V}$, where $\overline{V}\nsubseteq\textup{Supp}(D')$ for some $x\geq 0$. Scheme-theoretically $C=\overline{V}\cap E$ is a degree $m$ curve in $E\simeq \PP^2$, so $\mult_{C}(\overline{D}|_E) \geq m_t+x$.
	
As $Y_{\bullet}$ is generic, $l$ and $P$ are not in $C$. So, if $R'=D'|_E$, the inequality and B\'ezout's theorem yield
	\[
	a+b \ \leq \ \textup{deg}(R') \ \leq \ t-m(m_t+x)\ .
	\] 
Applying the definition of Newton-Okounkov bodies, this leads to a proof.
\end{proof}

\subsection{Infinitesimal data on singular surfaces}

Later it becomes important to make use of the infinitesimal picture for singular spaces for induction arguments. We explain this for surfaces. 

Let $S$ be an irreducible projective surface, $M$ an ample line bundle on $S$ and $x\in S$ a smooth point. Let $\pi :S'\rightarrow S$ be the blow-up at $x$ with $E\simeq \PP^1$ the exceptional divisor, which is Cartier. Note that the definitions of $\epsilon(M,x)$ and $\mu(M,x)$ make sense. So, let's consider the fiber-product diagram:
\[ \begin{tikzcd}
\overline{S}' \arrow{r}{\overline{\pi}} \arrow[swap]{d}{f'} & \overline{S} \arrow{d}{f} \\%
S' \arrow{r}{\pi}& S
\end{tikzcd}
\]
where $f:\overline{S}\rightarrow S$ is a resolution of singularities of $S$ that is an isomorphism in a neighborhood of $x$. Note that $\overline{\pi}:\overline{S}'\rightarrow \overline{S}$ is the blow-up of the same point $x\in\overline{S}$.
\begin{lemma}\label{lem:infinitesimalwidth}
	Under the assumptions above, $\mu(M,x)=\mu(f^*(M),x)$ and $\epsilon(M,x)=\epsilon(f^*(M),x)$.
\end{lemma}
\begin{proof}
	The equality between the Seshadri constants follows from the fact that in their definitions it does not matter whether the class $M_t\deq \pi^*(M)-tE$ is ample or nef.
	
	Now, let $Y_{\bullet}: E\ni z$ be a flag on $S'$, and denote the same way as a flag on $\overline{S}'$. We want to prove that 
	\[
\Delta_{Y_{\bullet}}\big(\pi^*(M)\big) \ = \ \Delta_{Y_{\bullet}}\big((f'\circ\pi)^*(M)\big)
	\]
	as sets in $\RR^2$, whose existence make sense since all the elements in the flag are smooth at $z$.
	
Effectivity for Cartier divisor is invariant under pull-back and since $f'$ is an isomorphism in a neighborhood of $E$, the direct inclusion follows. By \cite[Theorem 2.3]{LM09} both of these sets have the same area equal to $(M^2)=(f^*(M)^2)$. So, these two bodies must coincide.
	
	It remains to show that the horizontal width of each convex set is the infinitesimal width of the corresponding line bundle. For $f^*(M)$ it is clear, as $\overline{S}$ is smooth. For $A$, we have
	\[
	\big(D\cdot E\big) \ = \ \big(M_t\cdot E\big) \ = \ \big((\pi^*(M)-tE)\cdot E\big) \ = \ t \ > \ 0  \ 
	\] 
	for an effective $\QQ$-Cartier divisor $D\equiv M_t$. But the formula in \cite[Theorem 2.8, Chapter VI]{Ko96} yields $\textup{Supp}(D)\cap E\neq \varnothing$. So, let $F\deq \pi_*(D)$, defined as the appropriate linear combination of the images of the irreducible components of $F$. Then $F\equiv M$, as $D$ is and $x\in S$ is a smooth point. Furthermore, $D=\pi^*(F)-tE$ as an equality of effective Cartier divisors. So, $\mu(M,x)$ the infinitesimal width can be described only by those divisors coming from downstairs and this ends the proof.
\end{proof}

\section{Infinitesimal picture of abelian threefolds}
We study the infinitesimal picture through Newton-Okounkov bodies of an abelian threefold $A$ at the origin $0\in A$ for an ample $\QQ$-class $B$ on $A$. There are a few things that work better in our setup. First, the origin behaves for positivity issues like a very general point. Second, the Seshadri constant and the infinitesimal width don't depend on the base point, due to the group structure. Third, the tangent bundle of $A$ is trivial, so we can use differentiating techniques to positivity. Finally, Debarre \cite{D94} shows that any curve of small degree with respect to $B$ is degenerate, if $B^3$ is big.

\subsection{Small remarks on abelian threefolds.}

Most of the times we don't know whether our log-canonical centers pass through the origin. But, when they are two-dimensional case we get lucky.
\begin{lemma}\label{lem:nakayama}
	Suppose for some $t\in(0,\mu)$ there is a surface $\overline{S}\subseteq \Bplus(B_t)$. Then $0\in S\deq \pi(\overline{S})$ and $S$ is either an abelian surface or $\mult_0(S)\geq 2$.
\end{lemma}
\begin{proof}
	Suppose $\overline{S}\cap E=\varnothing$ and tweaking $t$ slightly assume by  Lemma~\ref{lem:baselocus} that all the base loci of the class $B_t$ agree. Then an effective divisor $\overline{D}\equiv B_t$ can be written $\overline{D}  = F +s\overline{S}$, where $\overline{S}\nsubseteq \textup{Supp}(F)$ and some $s>0$. As $0\notin S$,  $x\notin \pi\big(\Bstable(B_t)\big)$ and $0\notin S_x\deq S-x$ for a general point $x\in A$. So,
	\[
	\overline{D}'\ \deq \ F+s\overline{S}_x \ \equiv \ B_t \ \textup{ and  } \overline{S}\ \nsubseteq \ \textup{Supp}(\overline{D}') \ .
	\]
Due to the equality of our base loci this leads to a contradiction as $\overline{S} \nsubseteq\Bstable(B_t)$.
	
Assume now that $\mult_0(S)=1$. Let $x\in S$ be a smooth point and $\overline{S}_x=\pi^*(S_x)-E$ the proper transform of $S_x$. Note that both $\overline{S}_x$ and $\overline{S}$ lie in the same numerical class on $\overline{A}$. So, as above if $\overline{S}\subseteq \Bplus(B_t)$ then $\overline{S}_x\subseteq \Bplus(B_t)$ for $t<\mu$. If $S$ is not abelian, the closure of the union of all $S-x$, when $x\in S$ is smooth, is $A$, contradicting that $B_t$ is big, as we chose $t<\mu$. 
\end{proof}

A classical result of Debarre says that small degree curves on abelian threefolds are degenerate.  
\begin{proposition}\label{prop:debarre}
	Let  $B$ be an ample $\QQ$-divisor class on $A$  with $(B^3)>2$. If $C\subseteq X$ is a curve with $(B\cdot C)\leq 2$ then either $C$ is elliptic or $S\deq C+C\subseteq X$ is an abelian surface with $(B^2\cdot S)\leq 2$.
\end{proposition}
\begin{proof}
	Debarre shows in  \cite{D94} that whenever $C+C+C=X$, we have  
	\[
	(B\cdot C ) \ \geq \ 3\cdot \sqrt[3]{B^3/6} \ .
	\] 
	This will not happen if  $(B\cdot C)\leq 2$ and $B^3>2$, therefore $C$ must be degenerate. 
	
	If $C$ is elliptic, we are done. So, let $S= C+C$ be abelian. Restricting to $S$ note $(C^2)\geq 2$, as $C\subseteq S$ is not elliptic. Then Hodge index implies $(B^2\cdot S) \cdot  (C^2) \leq (B\cdot C)^2$ and so $(B^2\cdot S)\leq 2$.
\end{proof}

\subsection{Differentials and infinitesimal base loci on abelian three-folds.}

Our goal is to study the behavior of the base loci and the multiplicity function along irreducible components on the blow-up of the origin. This part is mainly based on ideas \cite{ELN94}, \cite{EKL95}, \cite{N96}, and \cite{N05}.

From Section~2, let $t_{\overline{Y}}$ be the starting point when $\overline{Y}$ appears in $\Bplus(B_t)$ and $m_{\overline{Y}}:[0,\mu]\rightarrow \RR_+$ be its multiplicity function, for a subvartiety $\overline{Y}\subseteq \overline{A}$. When $Y\subset A$, then $t_Y$ and $m_Y$ are data computed on the proper transform $\overline{Y}$ of $Y$ on $\overline{A}$.

Nakamaye showed that for any subvariety $Y\subseteq X$ the function $m_Y$ has slope at least one, see \cite{N05}*{Lemma~1.3} for the general case or \cite{N96}*{Lemma~3.4} for abelian manifolds. We show that the same phenomenon happens for any subvariety $\overline{Y}\subseteq X'$ not only those coming from downstairs: 
\begin{proposition}\label{prop:nakamaye}{(\textbf{Differentiation techniques}).}
	With the above notation let $t_1\in [t_{\overline{Y}},\mu)$. Then
	\[
	m_{\overline{Y}}(t)\ -\ m_{\overline{Y}}(t_1) \ \geq \ t-t_1\ \ \forall t\geq t_1 \ .
	\]
\end{proposition} 
\begin{proof}
	We follow the path set in \cite{ELN94}. To explain, fix $M$ an ample line bundle on an abelian manifold $Z$ of dimension $n$. The sheaf of differential operators of order $k\geq 1$ with respect ot $M$, denoted by $\mathbf{\mathscr{D}}^k_M$, has two important properties. First, \cite[Lemma 2.5]{ELN94} for abelian manifolds yields that $\mathbf{\mathscr{D}}_{L}^k\otimes \sO_Z((3n+3)M)$ is globally generated for any line bundle $L$ on $Z$ and any $k\geq 1$. Second, for an effective divisor $R\in |L|$, there is a natural morphism of vector bundles $\mathbf{\mathscr{D}}_{L}^k\rightarrow\sO_Z(L)$, which locally takes any differential operator $D$ of order $\leq k$ to $D(f)$ where $f$ is the local function of $R$.
	
On our abelian three-fold $A$ let $M$ be an ample line bundle and $R\in |mM|$ an effective divisor for some $m\geq 1$. Then we can build a subspace of sections
	\begin{equation}\label{eq:differentiation}
		0 \ \neq \ V_{k,m} \ \subseteq \ H^0\big(A, \sO_A((m+12)M)\big) \ ,
	\end{equation}
where each section is obtained by differentiating $R$ with a differential operator of order at most $k$.
	
Taking an asymptotic view of $(\ref{eq:differentiation})$, then for any effective $\QQ$-divisor $D\equiv B$, any positive rational numbers $0<\epsilon\ll\rho$, there exists a new effective $\QQ$-divisor $\partial^{\rho}(D)=(1+\epsilon)B$ constructed by applying to $mD$, for some $m\gg0$ and divisible enough, a differential operator of order $\leq m\rho$, and then normalizing by $m$. Note that the choice of $\epsilon$ and $\rho$ are independent. So, fixing $\rho$ and taking $\epsilon\rightarrow 0$, we can consider $\partial^{\rho}(D)\equiv B$ for our computations.
	 
Multiplicity behaves linearly with respect to differentiation. So, if $0\neq Y\subseteq A$ is a subvariety, then
	\[
	\mult_Y(D) \ - \ \mult_Y(\partial^{\rho}(D)) \ \geq \ \rho
	\]
These ideas and basic properties of asymptotic multiplicity show the statement in this case.
	
	For an infinitesimal subvariety $\overline{Y}\subseteq E$ we need a more careful analysis. Let $\{u_1,u_2,u_3\}$ be a local system of parameters at the origin $0\in X$. Locally the effective Cartier divisor $R\geq 0$ is given
	\[
	P_s(u_1,u_2,u_3) \ + \ P_{s+1}(u_1,u_2,u_3) \ +\ \ldots 
	\]
	where each $P_i\in\CC[u_1,u_2,u_3]$ are homogeneous polynomials of degree $i$ and $s=\mult_0(R)$. If $\overline{R}$ is the proper transform of $R$ on the blow-up $\overline{A}$, then
	\[
	\overline{R}|_E \ =  \ \textup{Zeroes}(P_s) \ \subseteq \ E\simeq \PP^2 \ .
	\]
	In this local setup, any differential operator of order at most $k$ is a homogeneous polynomial of degree at most $k$ in the order one differential operator $\partial/\partial u_1, \partial/\partial u_2,\partial/\partial u_3$. So, then for any differential operator $\partial^k$ of order at most $k$ the proper transform $\overline{\partial^k(R)}|_E =\partial^k(P_s)$, if  $s>k$. In particular, all the above ideas yield the same inequality for the infinitesimal subvarieties.
\end{proof}

\subsection{Sums of subvarieties and infinitesimal base loci.}
Applying Proposition~\ref{prop:nakamaye} and the group structure on abelian manifolds, we can manage when new sums of subvarieties show up in $\Bplus(B_t)$. 

Let $Y_1,Y_2\subseteq A$ be two subvarieties, such that $t_{Y_1}\leq t_{Y_2}< \mu(B;0)$. We consider the addition map
\[
(y_1,y_2)\in Y_1\times Y_2\ \rightarrow \ y_1+y_2\in X \ .
\] 
The image, denoted by $Y_{12}\deq Y_1+Y_2\subseteq A$, will then be an irreducible subvariety of $A$. The following example studies a few cases of sums of subvarieties that will be helpful later.
\begin{example}[\textbf{Sums of subvarieties of abelian threefolds}]\label{ex:sums}
\begin{enumerate}
	\item Let $Y_1\subseteq A$ be a non-degenerate curve, i.e. $Y_1+Y_1+Y_1=A$. \cite[Example 2.6, Corollary 2.7, Chapter VIII]{D99} yield $Y_{12}\subseteq A$ is a surface and $Y_1+Y_1+Y_2=A$ for any curve $Y_2\subseteq A$. 
	\item If $Y_1,Y_2$ and $Y_{12}$ are curves, then there is an elliptic curve $F\subseteq A$ and $x_1,x_2\in A$ with $x_1+F=Y_1$ and $x_2+F=Y_2$. See this, by translating $Y_1,Y_2$ to pass through $0$ and use irreducibility of $Y_{12}$.
	\item Let $0\in Y\subseteq X$ be a non-abelian surface. As $Y\pm Y$ are irreducible, containing $Y$, then $Y\pm Y=X$.
	\end{enumerate}
\end{example}
Proposition \ref{prop:nakamaye} yields a criterion for the proper transform of $Y_{12}$ to show up in $\Bminus(B_t)$.
\begin{lemma}\label{lem:uppermu}
	With the above notation, suppose the origin $0\in Y_1$. Then
	\[
	t_{Y_{12}}\ \leq \ t_{12} \deq \ \textup{min}\{t>0\ | \ m_{Y_2}(t)\geq t_{Y_1} \}.
	\]
\end{lemma}
\begin{proof}
	Fix a rational number $t'>0$ with $m_{Y_2}(t')> t_{Y_1}$. We have the same notation for points distinct from the origin both on $A$ and $\overline{A}$. It suffices to show $\overline{Y}_1^p  \subseteq \Bminus(B_{t'})$ for any $0\neq p\in Y_2$, where $\overline{Y}_1^p$ is the proper trasnform of $Y_1^p= Y_1+p$.  Furthermore, as $0\in Y_1$, we have $p\in Y_1^p$.

First, the condition $m_{Y_2}(t')> t_{Y_1}$ yields $\mult_{\overline{Y}_2}(\overline{D}) > t_{Y_1}$, for any effective $\QQ$-divisor $\overline{D}\equiv B_{t'}$, due to Lemma~\ref{lem:multiplicity}. So, using the definition of the multiplicity along a subvariety we then have
	\[
	\textup{mult}_{p}(\overline{D}) \ >\  t_{Y_1}\ \textup{ for all } \overline{D}\equiv B_{t'} \ ,\textup{ for any } p\in Y_2 \ .
	\]
Let $\pi_p:A_p\rightarrow A$ be the blow-up of $A$ at $p$, with $E_p$ the exceptional divisor and $B_{t}^p\deq\pi_p^*(B)-tE_p$ for any $t\geq 0$. As $A$ is abelian, any effective divisor can be moved to any point. So, any effective divisor in the class $B_t^p$ is the pull-back by $T_p:A\rightarrow A$, where $T_p(x)=x-p$, of an effective divisor from $B_t$. Thus, the proper transform of $Y_1^p$ by $\pi_p$ is in $\Bminus(B_{t}^p)$, whenever $t\geq t_{Y_1}$.
	
	As $\textup{mult}_{p}(\pi_*(\overline{D})) > t_{Y_1}$, the proper transform of $\pi_*(\overline{D})$ by $\pi_p$ lies in the class of $B_{t_0}^p$, where $t_0=\mult_p(\overline{D})$. Thus, $Y_1^p\subseteq\textup{Supp}(\pi_*(\overline{D})$ and $\mult_{Y_1^p}(\pi_*(\overline{D})) > 0$. This containment holds for any $\overline{D}\equiv B_t$. So, $\overline{Y}_1^p \subseteq \Bstable(B_{t'})$ and $\overline{Y}_1^p \subseteq \Bminus(B_{t'})$ for any $p\in Y_2$, by Lemma~\ref{lem:baselocus}.
\end{proof}

Based on the above Lemma~\ref{lem:uppermu} and Proposition~\ref{prop:nakamaye}, we can deduce easily the following corollary.
\begin{corollary}[\textbf{Bounds on infinitesimal width}]\label{cor:width} Under the above assumptions we have:
	\begin{enumerate}
		\item Let $0\in Y\subseteq A$ be a non-degenerate curve. Then $t_{Y+Y}\leq 2t_{Y}$ and $\mu(B,0)\leq 3\cdot t_Y$.
		\item If $Y_1,Y_2\subseteq A$ two subvarieties as in Lemma~\ref{lem:uppermu} so that $Y_1+Y_2=A$, then 
		\[
		\mu(B,0)\ \leq \  \ \textup{min}\{t>0\ | \ m_{Y_2}(t)\geq t_{Y_1} \}.
		\]
	\end{enumerate}
\end{corollary}
\begin{proof}
	$(2)$ follows from Lemma~\ref{lem:uppermu} and $(1)$ from Lemma~\ref{lem:uppermu} and Proposition~\ref{prop:nakamaye}.
\end{proof}

\subsection{Infinitesimal Newton--Okounkov bodies on abelian threefolds}
Applying the ideas from subsection~3.1 and those from Proposition~\ref{prop:nakamaye}, we describe here strong conditions on the shapes of infinitesimal Newton-Okounkov bodies for abelian threefolds for generic flag.
\begin{proposition}\label{prop:decreasing}
	Let $A$ be an abelian threefold, $B$ an $\QQ$-ample class and $Y_{\bullet}$ be a generic infinitesimal flag. If $\Bplus(B_{t_0})\cap E$ is one dimensional,  for some $t_0< \mu$, then the function 
	\[
	t \in (t_0,\mu)\ \ \longrightarrow \ \ \textup{vol}_{\RR^2} \Big(\tilde{\Delta}_{Y_{\bullet}}(B) \ \cap \ \{t\}\times \RR^2\Big)\ \in \RR \ \  is \ decreasing \ .
	\]
\end{proposition}
\begin{proof}
In order to prove that our function is decreasing, it suffices to show the implication
	\[
	\forall \ (t,a,b)\in \textup{Int}(\tilde{\Delta}_{Y_{\bullet}}(B)_{t}) \ \Longrightarrow \ 	(t,a,b) -(r,0,0) \ \in \ \textup{Int}\big(\tilde{\Delta}_{Y_{\bullet}}(B)\big)\ ,
	\]
 for any $0<r< t-t_0$, where $\textup{Int}(\Delta)$ is the topological interior of a convex set $\Delta\subseteq \RR^3$. Moreover, as the valuative points of $\tilde{\Delta}_{Y_{\bullet}}(B)$ form a dense subset, it suffices to show the above implication when $t,a,b,r\in \QQ$ and there exists a $\QQ$-effective divisor $D\equiv B$ such that $\nu_{Y_{\bullet}}(\pi^*(D))=(t,a,b)$.
	
	The idea is to show that there is a differential operator $\partial^r$ of order $r$, such that 
	\begin{equation}\label{eq:operator}
	\nu_{Y_{\bullet}}\big(\pi^*(\partial^r(D) \big) \ = \ (t-r,a,b) \ ,
	\end{equation}
where $\partial^r$ and $\partial^r(D)$ are denoted as in the proof of Proposition~\ref{prop:nakamaye}. Even if $\partial^r(D)\not\equiv B$, this suffices, as we can apply a limiting process, as explained in Proposotion~\ref{prop:nakamaye}, implying $(t-r,a,b)\in\tilde{\Delta}_{Y_{\bullet}}(B)$.
	
	To show $(\ref{eq:operator})$, we can assume that all the data is integral: $r,t,a,b\in \ZZ$ and $D$ is a Cartier divisor. Let $\{u_1,u_2,u_3\}$ be local coordinates of the origin and we assume $Y_{\bullet}|_E=\{\{u_1=0\}\ni [0:0:1]\}$. In particular, our effective divisor $D$ can be written locally as a Taylor series
	\[
	P_t(u_1,u_2,u_3) \ + \ P_{t+1}(u_1,u_2,u_3) \ +\ \ldots 
	\]
	where each polynomial $P_i$ is homogeneous of degree $i$. Since $\nu_{Y_{\bullet}}(\pi^*(D))=(t,a,b)$, then there exists two polynomials $Q\in\CC[u_1,u_2,u_3]$ and $R\in\CC[u_2,u_3]$, satisfying the properties
	\[
	P_t\ = \ u_0^a\cdot Q_t, Q_t(0,u_2,u_3)=u^b_1\cdot R_t(u_2,u_3), u_0\nmid Q_t\textup{ and } u_1\nmid R_t
	\]
	Taking all these properties into account we can then write 
	\[
	P_t(u_1,u_2,u_3) \ = \ u_1^a\cdot u^b_2\big(c\cdot u_3^{t-a-b}+R'_t(u_2,u_3)\big)+u_1^{a+1}R''_t(u_1,u_2,u_3) \ ,
    \]
	where $c\neq 0$, $R'$ is of degree $t-a-b$ and $R''$ of degree $t-a-1$. Finally, Theorem \ref{thm:propertiesnobody} implies $a+b+r<t$, and thus $\partial^r/\partial u_2^r(P_t)$ contains the monomial $u_1^au_2^bu_3^{t-a-b-r}$ with a non-zero coefficient.
	
Setting now $\partial^r=\partial^r/\partial u_2^r$, then $\partial^r(D)$ satisfies $(\ref{eq:operator})$ and we finish the proof.
\end{proof}
Inspired by \cite{CN14}, we obtain the following theorem, as a consequence of the material above.
\begin{theorem}\label{thm:abelian}{(\textbf{Bounds on the infinitesimal Newton--Okounkov bodies}).}
	Let $A$ be an abelian $3$-fold, $B$ be an ample $\QQ$-divisor, and $Y_{\bullet}$ a generic infinitesimal flag. Suppose $C\subseteq A$ is a curve with $q\deq\mult_0(C)\geq 1$, so that $\mu(B,0)>t_{C}$. Then for any $t\geq t_C$, the following inequality holds
	\[
	\textup{vol}_{\RR^2} \Big(\tilde{\Delta}_{Y_{\bullet}}(B) \ \cap \ \{t\}\times \RR^2\Big) \ \leq \ V_C(t)\ \deq \ \begin{cases} \frac{t^2-q(t-t_C)^2}{2}, \textup{ when } t_C\leq t\leq \frac{q}{q-1}t_C\\
	                                  \frac{q}{q-1}\cdot \frac{t_C^2}{2}, \textup{ when } t\geq \frac{q}{q-1}t_C
	\end{cases}
	\]
\end{theorem}
\begin{proof}
Fix a rational number $t\in (t_C,\mu(B;0))$, so that Lemma~\ref{lem:baselocus} holds. As any rational vector in $\textup{Int}\big(\tilde{\Delta}_{Y_{\bullet}}(B)\big)$ is valuative, then $E\nsubseteq \Bplus(B_t)$. So, \cite[Example 4.22 and Theorem 4.24]{LM09} yield
	\[
	\textup{vol}_{X'|E}(B_t) \ = \ 2 \cdot \textup{vol}_{\RR^2} \Big(\tilde{\Delta}_{Y_{\bullet}}(B) \ \cap \ \{t\}\times \RR^2\Big) \ .
	\]
The left side is the restricted volume from \cite{ELMNP06}, and can be described geometrically as follows
\[
	\textup{vol}_{X'|E}(B_t) \ = \ \lim_{m\rightarrow\infty} \frac{\#\Big(E\cap D_{m,1}\cap D_{m,2}\setminus\Bstable(B_t)\Big)}{m^2} 
	\]
	where $D_{m,1},D_{m,2}\in |mB_t|$ are two general choice of divisors and the limit is considered for those $m>0$ whenever the definitions and data make sense.
	
We consider first the case when $\Bplus(B_t)\cap E$ is one-dimensional for any $t\in (t_{C}, t')$. Then the function $t\rightarrow \textup{vol}_{X|E}(B_t)$ is decreasing, by Proposition~\ref{prop:decreasing}. Since $V_C$ is increasing, then 
	\[
	\textup{vol}_{X'|E}(B_t) \ \leq \ 2V_C(t) \ , \forall \ t\in (t_C,t')\ ,
	\]
The second case to consider is when $\Bplus(B_t)$ is one dimensional and has no irreducible components in $E$ for any $t\in [t_C, t')$. In this setup, then Proposition~\ref{prop:nakamaye} yields
	\[
	\mult_{\overline{C}}\big(||B_t||\big) \ \geq \ t-t_C \ \textup{ for any } t\in (t_C,t').
	\]
	
For any $m\gg 0$ and divisible enough, let $D_{m,1},D_{m,2}\in|mB_t|$ be general divisors. The above properties and Bertini's theorem, implies that $D_{1,m}\cap D_{2,m}$ is one dimensional and has no components in $E$. So, as a cycle we can write
	\[
	D_{1,m}\cap D_{2,m} \ = \ \big(m^2(t-t_C)^2+d\big)\overline{C} \ + \ F_m \ ,
	\]
	for some $d\geq 0$, such that $\overline{C}\nsubseteq \textup{Supp}(F_m)$. Under these conditions we have the inequality
	\[
	m^2\cdot t^2 \ = \ (D_{m,1}\cdot D_{m,2}\cdot E) \ = \ \Big(\big(m^2(t-t_C)^2+d\big)\overline{C} \ + \ F_m\cdot E\Big)\ \geq  \ q(m^2(t-t_C)^2) + (F_m\cdot E) \ .
    \]
    Since $\Bplus(B_t)=\Bstable(B_t)$ by our assumption above, this inequality yields finally
    \[
    \textup{vol}_{X'|E}(B_t) \ = \ \lim_{m\rightarrow} \frac{\#\Big(E\cap D_{m,1}\cap D_{m,2}\setminus\Bstable(B_t)\Big)}{m^2} \ \leq \ \lim_{m\rightarrow}\frac{(F_m\cdot E)}{m^2} \ \leq \ V_C(t) \ = \ t^2-q(t-t_C)^2 \ .
    \]
This finishes the proof, by making use of continuity and convexity.
    \end{proof}

\subsection{An inductive approach}
In our setup we need to deal with the case when there is an abelian surface $S\subseteq A$. Inversion of adjunction techniques allows us to produce singular divisors on $X$ from those on $S$. The possibility to do so on $S$, was originally tackled in \cite{KL19} and \cite{I18}.
\begin{proposition}\label{prop:inductive}{(\textbf{Inductive process and singular divisors})}
	Let $A$ be an abelian three-fold and $S\subseteq A$ an abelian surface. Suppose $B$ is an ample $\QQ$-divisor on $A$, satisfying the properties:
	\begin{enumerate}
		\item $(B^2\cdot S) \ > \ 4$.
		\item For any elliptic curve $C\subseteq S$ we have $(B\cdot C)>1$.
    \end{enumerate}
    If $B-S$ is ample then there is an effective $\QQ$-divisor $D\equiv B$ with $\sJ(X;D)=\m_{A,0}$. 
\end{proposition}

\begin{proof}
Rigidity lemma says that any morphism between two abelian varieties is the composition of a translation and a group homomorphism. For us this means that there is a point $x\in A$ and an abelian surface $S_0\subseteq A$, where the embedding preserves the group structure, so that $S=S_0+x$. 
	
Moreover, this forces $(S_0+y)\cap S_0=\varnothing$ for any $y\notin S_0$. So, we have
	\begin{equation}\label{eq:intersection} 
	(B\cdot S^2) \ = \ (S^3)\ = \ (S\cdot C) \ = \ 0 \ , \textup{ for any curve } C\subseteq S \ .
	\end{equation}
	In order to understand how to construct $D$, write the class $B$ as follows
	\[
	D\ \equiv \ \Big(\big((1-\epsilon)B-S\big) \ + \ (1-\delta)S\Big)\ + \ \big(\epsilon B+\delta S\big) \ = \ B \ , \textup{ for any } 0<\delta \ll \epsilon \ll 1 \ .
	\] 
	The goal is to construct out of each term an effective divisor with certain properties.
	
	First, let $0<\epsilon\ll 1$ be a rational number, so that $(1-\epsilon)B$ satisfies $(1)$ and $(2)$ in the statement and $B^{\epsilon}=(1-\epsilon)B-S$ is ample. By $(\ref{eq:intersection})$ $B^{\epsilon}$ also satisfies $(1)$ and $(2)$. So, applying \cite[Proposition 3.1]{I18} and Lemma~\ref{lem:trick}  to $B^{\epsilon}|_S$, then there is an effective $\QQ$-divisor $D_1^S\equiv B^{\epsilon}|_S$ with 
	\[
	\sJ(S,(1-c)D_1^S) \ = \ \m_{S,0}, \ \forall \ 0\leq c\ll 1 \ .
	 \]	
Next we use the ideas from the proof of Proposition~\ref{prop:cuttingdown1}. Let $m\gg 0$ and divisible enough so that $mD_1^S$ on $S$ and $mB^{\epsilon}$ on $A$ are Cartier and very ample, the map
$H^0(A,\sO_A(mB^{\epsilon}))\rightarrow H^0(S,\sO_S(mB^{\epsilon}))$
	 is surjective, and $\sI_{S|A}(mB^{\epsilon})$ is globally generated. These assumptions force the linear series
	 \[
	 |V^{\epsilon}_m| \ \deq \ \{D^{\epsilon}_m\in|mB^{\epsilon}|, \textup{ where } D^{\epsilon}_m|_S=mD_1^S\}
	 \]
	 to have no base points on $A\setminus S$. In particular, if $D^{\epsilon}_1=\frac{1}{m}D^{\epsilon}_m$ for some general $D^{\epsilon}_m\in|V^{\epsilon}_m|$, the ideal $\sJ(A,D^{\epsilon}_1)$ must be trivial on $A\setminus S$. As $D^{\epsilon}_1|_S=D_1^S$, then \cite[Example 9.5.16]{PAG} yields the inclusions
	 \[
	 \sJ(S,D_1^S) \ \subseteq \ \sJ(A,D^{\epsilon}_1+(1-\delta)S)\otimes \sO_S \ \subseteq \ \sJ(S,(1-c_1)D_1^S) \
	 \]
	 for some $0<c_1\ll 1$ and any $0<\delta \ll 1$. This in turn forces $\textup{LC}(D_1+(1-\delta)S)$ to be zero-dimensional. By tweaking our divisor, its multiplier ideal becomes equal to $\m_{A,0}$, by Lemma~\ref{lem:ito}. As $\epsilon B+\delta S$ is ample, then a general effective $\QQ$-divisor $D_2\equiv \epsilon B+\delta S$ would satisfy the property
	 \[	 
	 \sJ(A,D^{\epsilon}_1+(1-\delta)S+D_2) \ = \ \sJ(A,D^{\epsilon}_1+(1-\delta)S) \ = \ \m_{A,0} \ ,
	 \]	 
by \cite[Example 9.2.29]{PAG}. This finishes the proof.
\end{proof}
As an application of Newton-Okounkov bodies, we get a nice first step towards Corollary~\ref{cor:main}.
\begin{proposition}\label{prop:abelian surface}
Let $S\subseteq A$ be an abelian surface and $B$ and ample $\QQ$-divisor. If $B-S$ is ample, then there is a divisor $D\equiv B$ with $\sJ(A,D)= \sO_A(-S)$. Otherwise, $3\cdot (B^2\cdot S)\geq B^3$.
\end{proposition}
\begin{proof}
	If $B-S$ is ample then the construction of a divisors as above is done in \cite[Example 9.2.29]{PAG}. Now, assume $B-S$ is not ample. Since our ambient space is abelian then every pseudo-effective divisor is nef. As a consequence, this yields the inequality
	\[
	1 \ \geq \ \rho \ \deq \ \textup{max}\{t>0\ | \ B-tS \textup{ is pseudo-effective}\} \ .
	\]
	Based on the theory of Newton--Okounkov bodies from \cite{LM09}, we have the identity
	\[
	\textup{vol}_A(B) \ = \ 3\cdot \int_0^{\rho} \textup{vol}_{A|S}(B-tS)\ dt \ .
	\]
As $B-tS$ is ample for any $t\in [0,\rho)$, then Serre vanishing and asymptotic Riemann Roch imply $\textup{vol}_A(B) = B^3$ and $\textup{vol}_{A|S}(B-tS) = (B^2\cdot S)$ for any $t\in[0,\rho)$, where the latter is due to $(\ref{eq:intersection})$. All this instead lead to the desired inequality in our statement.
	\end{proof}
We don't have nice numerical criteria, as above, when the LC locus of a divisor is a non-abelian surface $S\subseteq A$.  Luckily, the infinitesimal width does behave well.
\begin{proposition}\label{prop:induction}{(\textbf{Induction vs. infinitesimal width})}
	Let $S\subseteq A$ be an irreducible projective non-abelian surface smooth at the origin $0\in A$. If $B$ is an ample $\QQ$-class on $A$, then $\mu(B,0) \leq \mu(B|_S, 0)$.
\end{proposition}
\begin{proof}
The infinitesimal width is homogeneous of degree one. So, it suffices to prove the statement when $B$ is a line bundle on $A$. The proof is done by contradiction, so assume $\mu(B,0)>\mu(B|_S, 0)$. Denoting $S_x= S-x$ for a very general point $x\in S$, the main idea is to show the inequalities
	\begin{equation}\label{eq:infinitesimalwidth}
		t_{S_x} \ \leq \ \mu(B|_S, x)\ \leq \ \mu(B|_S,0) \ .
	\end{equation}
	Let's first see how they lead to a contradiction.	Since $\mu(B|_{S_x},0) =\mu(B|_S,x)$, they yield
	\[
	\overline{S}_x \ \subseteq \Bplus(B_t) \ , \ \forall  t\ \in \big(\mu(B|_S,0),\mu(B,0)\big) \ , 
	\]
	where $\overline{S}_x$ is the proper transform of $S_x$. As $S-S=X$, the closure of the union of $S_x$, with $x\in S$ very general, is $A$. Thus, $\Bplus(B_t)=A$ for any $t>\mu(B|_S,0)$, contradicting our assumption.
	
Let's prove $(\ref{eq:infinitesimalwidth})$ and we start with the first inequality. Choose a rational number $t_0>\mu(B|_{S_x},0)$ and let $D=\pi_{S_x}^*(B|_{S_x})-t_0E_{S_x}$, where $\pi_{S_x}:\overline{S}_x\rightarrow S_x$ is the blow-up of $S_x$ at $0$ and $E_{S_x}$ is the exceptional divisor. For any $m\gg 0$ and divisible enough consider the following exact sequence
	\[
	0\rightarrow H^0\big(\overline{A},\sO_{\overline{A}}(mB_{t_0}-\overline{S}_x)\big)\rightarrow H^0(\overline{A},\sO_{\overline{A}}\big(mB_{t_0})\big)\rightarrow H^0\big(\overline{S}_x,\sO_{\overline{S}_x}(mD)\big)
	\] 
	By the definition of $\mu(B|_{S_x},0)$ the group on the right vanishes. Thus $\overline{S}_x\subseteq \Bstable(B_{t_0})\subseteq\Bplus(B_{t_0})$ and the left inequality in $(\ref{eq:infinitesimalwidth})$ follows immediatly from the definition of the invariant $t_{S_x}$.
	
	For the second inequality, by Lemma~\ref{lem:infinitesimalwidth}, we can assume that $S$ is smooth everywhere. So, let $\pi_{S}: Y= \textup{Bl}_{\Delta_S} \rightarrow S\times S$	be the blow-up of the diagonal $\Delta_{S}\subseteq S\times S$ with $E_{\Delta}\subseteq Y$ the exceptional divisor. Let $\pi_1,\pi_2:Y\rightarrow S$ be the projection morphisms to each corresponding factor.
	
	We study the flat family $\pi_1:Y\rightarrow \overline{A}$, where each $\pi_1^{-1}(x)$ is the blow-up of $S$ at $x$. Let $\mathcal{B}=\pi_2^*(B)$. Note that $\mathcal{B}|_{\pi_1^{-1}(x)} =\pi_x^*(B)$, where $\pi_x=\pi_2|_{\pi_1^{-1}(x)}:\pi_1^{-1}(x)\rightarrow S$ is the blow-up morphism, with the exceptional divisor $E_x=E_S\cap \pi_1^{-1}(x)$. On $Y$ consider the $\QQ$-Cartier divisor $\mathcal{B}-t\cdot E_S$, for any $t\in \QQ_+$. If $(\mathcal{B}-t\cdot E_S)|_{\pi_1^{-1}(x')}$ is $\QQ$-equivalent to an effective divisor for some $x'\in S$ and $t\in \QQ_+$, then by \cite[Theorem III.12.8]{Har} the same holds either for all $x\in S$ or those lying on  a countable union of curves and points of $S$. By definition, this implies the second inequality in $(\ref{eq:infinitesimalwidth})$.
\end{proof}

The connection between the Seshadri constant and the infinitesimal width runs deeper than Corollary \ref{cor:width}. Let $C\subseteq A$ be an irreducible curve with $q\deq \textup{mult}_0(C)\geq 2$ and $\epsilon(B,0) =  (B\cdot C)/q$, i.e. a \textit{Seshadri exceptional curve} of $B$. Inspired by \cite[Lemma 4.2]{CN14}, we have the following proposition:
\begin{proposition}\label{prop:seshadriwidth}
	Under the above notation, assume $S=C+C\subseteq A$ is not an abelian surface. Then at least one of the two conditions is satisfied:
	\begin{enumerate}
		\item $\mu(B, 0)\ \leq \ \frac{q}{q-1}\epsilon(B,0)$. 
		\item $\textup{mult}_C(S)\ \geq \ 2 $.
	\end{enumerate}
\end{proposition}
\begin{remark}\label{rem:abelian}
When $S$ is abelian, we have a slightly different inequality
	\[
	t_S \ \leq \frac{q^2}{q^2-q+2}\epsilon(B,0) \ .
	\]
First, the first inequality in $(\ref{eq:infinitesimalwidth})$ works when $S$ is abelian, so $t_S\leq \mu(B|_S;0)$. Second, for the proof of Proposition~\ref{prop:seshadriwidth} we uses the inequality $C^2\geq q^2-q$ from \cite{EL93}, holding for any surface. If $S$ is abelian \cite{KSS09} provides a stronger one $C^2\geq q^2-q+2$.
\end{remark}
\begin{proof}
	Suppose that $S$ is smooth at a very general point $x\in C$. Applying Proposition~\ref{prop:induction} and the fact that $C_x\deq C+x$ has the same data at $x$ as $C$ at the origin, then it suffices to show
	\[
	\mu(B|_S,x) \ \leq \ \frac{q}{q-1}\epsilon(B|_S,x) \ .
	\] 
Modulo a resolution of singularities, we can assume by Lemma~\ref{lem:infinitesimalwidth} that $S$ is everywhere smooth and $R=B|_S$ is big and nef. If $F\subseteq S$ is a curve with $(R\cdot F)=0$, then $F^2<0$ by Hodge index. So, the choice of $x$ implies that there is no such curve through it. Our goal is to show
	\[
	\epsilon(R;x)\cdot \mu(R;x) \ \leq \ (R^2) \ .
	\]
First, let's see how it implies our statement. As $q\geq 2$, $C_x$ moves in a non-trivial family on $S$, even after passing to a resolution of $S$. Now, \cite[Corollary 1.2]{EL93} yields $(C_x^2)\geq q^2-q$. The above inequality and Hodge index would finish the proof, as they imply the following sequence of inequalities
		\[
		\mu(R,x) \ \leq \ \frac{R^2}{\epsilon(R,x)}\ \leq \ \frac{(R\cdot C)^2}{C^2\cdot \epsilon(R,x)} \ \leq \ \frac{(R\cdot C)^2}{(q^2-q)\epsilon(R,x)} \ \leq \ \frac{q}{q-1}\epsilon(R,x) \ ,
		\]
The above inequality can be best seen through the theory of infinitesimal Newton--Okounkov polygons. \cite[Theorem D]{KL18} shows that any such convex set contains the triangle with the vertices $(0,0),(\epsilon(R,x),0)$ and $(\epsilon(R,x),\epsilon(R,x))$. Since $\mu(R,x)$ is the horizontal width of any of these sets, whose areas is $(R^2)/2$ by \cite[Theorem A]{LM09}, this implies the inequality.
\end{proof}

\section{Syzygies vs. singular divisors on abelian manifolds}
In this section we discuss the connection between syzygies and singular divisors on abelian manifolds. The property $(N_p)$ was considered in \cite{LPP11}. For the vanishing of $K_{p,1}(X,L;dL)$ for any $d\geq 1$, we use the ideas from \cite{ELY16} and \cite{LPP11}, to show that on any abelian manifold the vanishing of these Koszul groups is related to singular divisors in the same fashion as for property $(N_p)$.
\begin{theorem}[\textbf{Syzygies vs. singular divisors}]\label{thm:lpp}
	Let $(A,L)$ be a polarized abelian variety and $p\geq 0$ an integer. Suppose there is an effective $\QQ$-divisor $F$ on $A$ satisfying the following properties:
	\begin{enumerate}
		\item $F\ \equiv \ \frac{1-c}{p+2}\cdot L$ for some $0<c<1$. 
		\item The multiplier ideal $\sJ(A,F)$ has zero-dimensional support . 
	\end{enumerate}
	Then $(A,L)$ satisfies property $(N_p)$ and $K_{p,1}(A;L,B) = 0$ for any line bundle $B$, with $B-L$ is ample.
\end{theorem}
\begin{remark}
	Using \cite[Remark 1.4 and 1.8]{ELY16} together with the proof of Theorem~\ref{thm:lpp}, we can further prove that under the same above conditions, $L$ is $p$-jet very ample. 
\end{remark}

\begin{proof}[Proof of Theorem~\ref{thm:lpp}]
	First, we tackle property $(N_p)$. By \cite{LPP11} it suffices to find a divisor 
	\begin{equation}\label{eq:specialdivisor}
	F'\equiv\frac{1-c}{p+2}\cdot L \textup{ with }\sJ(A;F')=\m_{A,0} \ . 
	\end{equation}
For $F$ as in the statement, Lemma~\ref{lem:ito} produces another divisor $F''$ with $\sJ(A;F'')=\m_{A,x}$ for some $x\in A$. Setting $F'=F''-x$, then $F'$ satisfies the conditions from \cite{LPP11}. So, property $(N_p)$ holds.
	
For the vanishing of $K_{p,1}$, we assume the existence of an effective divisor $F'$ satisfying $(\ref{eq:specialdivisor})$. Going forward, let $Y=A\times A^{\times (p+1)}$ and $\pi_{ij}:Y\rightarrow A\times A$ the projection of $Y$ to the $i$ and $j$ factor. Let $\Delta_{ij}=\pi_{ij}^*(\Delta)$, where $\Delta\subseteq A\times A$ is the diagonal. Now, define by 
	\[
	Z \ = \ \Delta_{01} \cup \ldots \cup \Delta_{0,p+1} \ \subseteq \ A\times A^{\times (p+1)} \ . 
	\]
First note that $H^1(A,B)=0$ by Kodaira's vanishing, as $B$ is ample and $K_A=0$. So applying \cite[Proposition 1.1]{ELY16}, we can prove that $K_{p,1}(A;L,B)=0$, as long as the restriction map
	\[
	H^0(Y,L\boxtimes B^{\boxtimes (p+1)}) \ \longrightarrow \ H^0(Y, (L\boxtimes B^{\boxtimes (p+1)})|_Z) \ \textup{ is surjective,}
	\]
	where $\boxtimes$ is the tensor product of the pull-backs by the projections to each factor. 
	
We get this by applying Nadel's vanishing \cite[Theorem 9.4.8]{PAG}. To do so, we construct an effective $\QQ$-divisor $E$  on $Y$ from $F'$ with $\sJ(Y;E)=\sI_{Z|Y}$ and $\big(L\boxtimes B^{\boxtimes (p+1)}\big)-E$ is ample. So,	let $\pi_i:A^{\times (p+1)}\rightarrow A$ be the projection to the $i$ factor and define $\delta : A\times A^{\times (p+1)}\rightarrow A^{\times (p+1)}$ by $\delta(x_0,\ldots ,x_{p+1})= (x_0-x_1,\ldots ,x_0-x_{p+1})$. If we define $E\deq \delta^*\Big(\sum_{i=1}^{p+1} \pi_i^*(F')\Big)$, then
	\[
	\sJ(Y,E) = \delta^*\Big(\sJ\Big(A^{\times (p+1)},\sum_{i=1}^{p+1} \pi_i^*(F') \Big)\Big) = \delta^*\Big(\prod_{i=1}^{p+1} \pi_i^*\big(\sJ(A,F')\big)\Big) = \delta^*\Big(\prod_{i=1}^{p+1} \pi_i^*\big(\m_{A,0}\big)\Big)  =  \sI_{Z|Y} .
	\]
	The first equlity follows from \cite[Example 9.5.45]{PAG}, as $\delta$ is a smooth morphism, the second one from \cite[Proposition 9.5.22]{PAG} and the last one by doing computations on a local level.
	
	It remains to show that $(L\boxtimes B^{\boxtimes (p+1)})- E$ is ample. \cite{LPP11}*{Proposition~1.3} yields the equalities 
	\[
	E\ \equiv_{\textup{num}} \ \frac{1-c}{p+2} \delta^*\Big(L^{\boxtimes (p+1)}\Big) \ = \ (1-c)L\boxtimes L^{\boxtimes (p+1)} - \frac{1-c}{p+2}N \ ,
	\]
	where $N$ is some nef class on $Y$. So, then we can write 
	\[
	L\boxtimes B^{\boxtimes (p+1)}-E \ \equiv_{\textup{num}} \ cL\boxtimes (B-L+cL)^{\boxtimes (p+1)} \ + \ \frac{1-c}{p+2}N \ .
	\]
The entire expression is ample, as the first summand is ample, because $B-L$ is so, and $N$ is nef.
\end{proof}

\section{Proofs of the main results}
This section is devoted to the proof of the main results of the paper. We will deal first with Theorem~\ref{thm:main} for higher syzygies, when $p\geq 0$, then deduce as a corollary the strongest possible criterion for base-point freeness, when $p=-1$, and finally finishing with a proof of Corollary~\ref{cor:main}

\subsection{Higher syzygies}

Let $(A,L)$ be a polarized abelian $3$-fold with $L^3> 59(p+2)^3$. Let $\pi:\overline{A}\rightarrow A$ be the blow-up of the origin $0\in A$. Denote by 
\[
B \ \deq \ \frac{1}{p+2} L \ ,
\]
by $B_t\deq\pi^*(B)-tE$ for any $t\geq 0$, $\mu\deq \mu(B,0)$ the infinitesimal width and $\epsilon\deq \epsilon(B,0)$ the Seshadri constant of $B$. Note here that  $\mu>3$, as $B^3\geq 59$.

For an effective $\QQ$-divisor $\overline{D}\equiv B_t$, with $E\nsubseteq\textup{Supp}(\overline{D})$,  its \textit{the push-forward} divisor $D\deq \pi_*(\overline{D})$ is defined as the linear span of the images of all irreducible components of $\overline{D}$ with the appropriate coefficients. In particular, $\overline{D}=\pi^*(D)-\mult_0(D) E$ as effective divisors.
\begin{proof}[Proof of Theorem~\ref{thm:main}]
	By Theorem~\ref{thm:lpp} it suffices to find an effective $\QQ$-divisor $D\equiv B$ so that $\textup{lct}(D)<1$ and $\textup{LC}(D)$ is zero-dimensional. For $t>3$ and any effective divisor $\overline{D}_t\equiv B_t$, not containing $E$ in its support,  the push-forward $D_t=\pi_*(\overline{D}_t)$ has $\textup{lct}(D_t)<1$, by \cite[Proposition 9.3.2]{PAG}. As $\mu>3$, then by Lemma~\ref{lem:trick} we can assume in the following that $D_t$ has all its critical subvarieties to be at least one-dimensional, for any effective divisor $\overline{D}_t\equiv B_t$ with $t>3$.

We divide the proof in a few steps. The first two use inversion of adjunction techniques. Here we either get our syzygetic properties or obtain upper bounds on the Seshadri constant $\epsilon(B,0)$. This and heavy computations on infinitesimal Newton--Okounkov bodies will lead to a contradiction.

\textbf{\textup{Step: 1}} Suppose that for some $\overline{D}_t\equiv B_t$, where $t\in (3,\mu)$, the proper push-forward $D_t$ has a smooth curve as a critical variety. Then either our syzygetic properties hold or $\epsilon(B;0)\leq 6-t$.

Let $C_t$ be the smooth curve and $\overline{C}_t$ its proper transform on $\overline{A}$.  By Lemma~\ref{lem:trick}, there is an effective divisor $D'_t\equiv c'_tB$ satisfying the properties that 
\[
\textup{Zeroes}\big(\sJ(X;D_t')\big) \ = \ C_t \textup{ and  } \sJ(X;(1-\delta)D_t') \ = \ \sO_X \textup{ for any } 0<\delta<1 \ .
\]
This divisor can be chosen as close as one wants to $\textup{lct}(D_t)\cdot D_t$, i.e. the multiplicity of $D_t'$ and $\textup{lct}(D_t')$ can be chosen to be arbitrarily close to those of $D_t$.

We now make use of inversion of adjunction techniques. Assume first the following inequality
\begin{equation}\label{eq:main1}
t\cdot c_t'\ + \ \epsilon\cdot(1-c_t') \ > \ 3 \ .
\end{equation}
Since $D_t'$ can be arbitrarily close to $D_t$, by Proposition~\ref{prop:cuttingdown2} there is an effective $\QQ$-divisor $D_t''\equiv B$ with zero-dimensional LC center, and by Theorem~\ref{thm:lpp} $(A,L)$ would satisfy both syzygetic properties.

Next, assume $(\ref{eq:main1})$ doesn't hold. This yields the following upper and lower bounds:
\begin{equation}\label{eq:main2}
c_t' \ \leq \ \frac{3-\epsilon}{t-\epsilon} \textup{ and } 1-c_t'\ \geq \ \frac{t-3}{t-\epsilon} \ . 
\end{equation}
With this in hand, we again use inversion of adjunction. Assume for this that 
\[
\frac{t-3}{t-\epsilon} (B\cdot C_t) \ > \ 1 \ .
\]
Combining it with $(\ref{eq:main2})$ and Proposition~\ref{prop:cuttingdown1}, we can construct again a new divisor $D_t''$ with $\textup{lct}(D_t'')<1$ and zero-dimensional LC-locus. So, again both syzygetic properties are satisfied.

It remains to study the case when the last inequality doesn't hold. 
If $(B\cdot C_t)\leq 2$, Proposition~\ref{prop:debarre} and the assumptions in the statement lead to a contradiction. Otherwise $(B\cdot C_t)>2$, and together with the inverse of the last inequality above easily imply $\epsilon(B,0) < 6-t$.

\textbf{\textup{Step 2:}} Fix $t\in (3,\mu)$ and suppose that for some divisor $\overline{D}_t\equiv B_t$ the proper push-forward $D_t$ has a surface $S_t\subseteq X$ as a critical variety. Then either we go back to Step~1 or $\textup{mult}_0(S_t)=2$.

By Step~1, we can assume that for any $\overline{D}_{t}\equiv B_{t}$, the critical variety of $D_t$ is a surface. By Lemma~\ref{lem:multiplier} and \ref{lem:multiplicity}, there is then a surface $\overline{S}_t\subseteq \Bplus(B_t)$, with its image $S_t$ being the log-canonical center of a general choice of $D_t$ with $\mult_0(D_t)=t$. Then \cite[Proposition 9.5.13]{PAG} yields
\begin{equation}\label{eq:boundmultiplicity2}
\mult_{\overline{S}_t}\big(||B_{t}||\big) \ \geq \ t/3 \ > \ 1 \ . 
\end{equation}
If $S_t$ is abelian, then $(\ref{eq:boundmultiplicity2})$ implies that $B-S_t$ is ample. So, applying Proposition~\ref{prop:inductive} and Theorem~\ref{thm:lpp}, we deduce that the pair $(A,L)$ satisfies the syzygetic properties.

If $S_t$ is not abelian, then Lemma~\ref{lem:nakayama} forces $\mult_0(S_t)\geq 2$. To prove equality, we apply $(\ref{eq:boundmultiplicity2})$ with Proposition~\ref{prop:nobodies} to the Newton--Okounkov body $\tilde{\Delta}_{Y_{\bullet}}(B)$, where $Y_{\bullet}$ is a general infinitesimal flag. Since $\mult_0(S)\geq 3$, this implies $\mu(B,0)\leq t$, contradicting our choice of $t\in(3,\mu)$.

\textbf{\textup{Step 3:}} The bad cases of Step~$1$ and $2$ contradict $B^3>59$. 

Our goal is to translate the "bad" cases from Step~$1$ and $2$ into strong conditions on the convex set $\tilde{\Delta}_{Y_{\bullet}}(B)$, for some fixed generic infinitesimal flag $Y_{\bullet}$. Step $1$ provides an upper bound on $\epsilon(B,0)$. Otherwise, for any $\overline{D}_t\equiv B_t$, the  LC-locus of $D_t$ contains a surface with multiplicity $2$ at the origin. Let $t_1\in (3,\mu]$ be where this starts happening. By Lemma \ref{lem:multiplier} and \ref{lem:multiplicity}, there is a surface $\overline{S}_1\subseteq\Bplus(B_{t_1})$, which works for general divisors. So, \cite[Proposition 9.5.13]{PAG} and Proposition~\ref{prop:nakamaye} imply 
  \begin{equation}\label{eq:main3}
  \textup{mult}_{\overline{S}_1}(||B_{t_1+r}||) \ \geq \ t_1/3+r, \textup{ for any }r\geq 0 \ .  
  \end{equation}
By Step $1$ and the definition of $t_1$, we also know the following upper bound
 \begin{equation}\label{eq:main4}
 t_1\ \leq \ 6 -\epsilon(B,0) \ .
 \end{equation}
Finally, besides $S_1$, the image of $\overline{S}_1$, there is a second surface $S_2$ that appears. If $\epsilon(B,0)$ is defined by a surface, then this is $S_2$. When it is defined by a non-elliptic curve $C\subseteq X$ with $q\deq \textup{mult}_0(C)\geq 1$, set $S_2\deq C+C$, otherwise there is no $S_2$. When $C$ is non-degenerate, Lemma~\ref{lem:uppermu} yields the inequalities
\begin{equation}\label{eq:main5}
t_{S_2} \ \leq \ 2\epsilon(B,0) \ \textup{and} \ \mu(B,0) \ \leq 3\epsilon(B,0) \ . 
\end{equation}
Using $(\ref{eq:main3})$, $(\ref{eq:main4})$, and $(\ref{eq:main5})$ we divide the proof in three cases based on the size of $\epsilon(B,0)$.  

\textbf{\textup{Case 3.1:}} $0\ < \ \epsilon(B;0) \ \leq  \ 1.5$.

We assume $\epsilon(B,0)$ is defined by a curve $C\subseteq A$, as the arguments below would still work by taking a curve with data very close to $\epsilon(B;0)$. Let's show first that our initial assumptions force $C$ to be non-degenerate and $q=\textup{mult}_0(C)\geq 5$. If $C$ is elliptic, the bounds on $\epsilon(B;0)$ yield $(B\cdot C)\leq 1.5$, which is not possible. If $S_2$ is abelian, \cite{KSS09} yields $C^2\geq q^2-q+2$. So Hodge index implies
\[
(B^2\cdot S_2)  \ \leq \ \frac{(B\cdot C)^2}{C^2} \ \leq \ \frac{q^2\cdot \epsilon(B,0)^2}{q^2-q+2} \ \leq \ (1.5)^2\cdot \frac{8}{7} \ < \ 4 \ ,
\]
again contrdicting the assumptions in the main statement. As $C$ is non-degenerate, the upper bound on $\epsilon(B,0)$, the inequality from \cite{D94}, used in the proof of Proposition \ref{prop:debarre}, and $B^3>59$ force $q\geq 5$. 

Now, we describe the upper bounds on the vertical slices of $\Delta_{Y_{\bullet}}(B)$. We use Theorem~\ref{thm:propertiesnobody} on $[0,\epsilon]\times\RR^2$ and Theorem~\ref{thm:abelian} on $[\epsilon,3\epsilon]\times\RR^2$  with $q=5$, as for higher multiplicity these bounds are stronger. Applying Theorem~\ref{thm:propertiesnobody} and Fubini's theorem yields the inequalities
\[
\frac{58}{6} \ < \ \textup{vol}_{\RR^3}\big(\tilde{\Delta}_{Y_{\bullet}}(B)\big)\ \leq \ \int_0^{\epsilon}\frac{t^2}{2} dt \ + \ \Big(\int_{\epsilon}^{\frac{5}{4}\epsilon} \frac{t^2-5(t-\epsilon)^2}{2} dt \ + \ \int_{\frac{5}{4}\epsilon}^{3\epsilon}\frac{5}{8}\epsilon^2 dt\Big) \  .
\]
The right side is $45\epsilon^3/32$ and the inequality cannot hold for $\epsilon\in (0,1.5)$, leading to a contradiction.

\textbf{\textup{Case 3.2:}} $1.5\ < \ \epsilon(B,0) \ <  \ 2$.

The same proof, as the start of Case~3.1, implies that our bounds on $\epsilon(B,0)$ force either the curve $C$ to be non-degenerate with $q=\textup{mult}_0(C)\geq 4$, or $S_2=C+C$ is an abelian surface with 
\[
4 \ < \ (B^2\cdot S_2) \ \leq \ 32/7 \ .
\]
Let's discard first the latter case. When $q=1$, Hodge index on $S_2$ does so, as $(B\cdot C)<2$, $(B^2\cdot S_2)\geq 4$ and $C^2\geq 2$, because $C$ is not elliptic. When $q\geq 2$, we have additionally the upper bounds
\[
t_{S_2} \ \leq \ (8\epsilon)/7\textup{ and } \mu(B,0) \ \leq \ (8\epsilon)/7+1 \ .
\]
The first one is due to Remark~\ref{rem:abelian} and the second is due to $\mu-t_{S_2}\leq 1$, as otherwise $B-S_2$ is ample and Proposition~\ref{prop:inductive} and Theorem~\ref{thm:lpp} imply our syzygetic properties. In particular, these upper bounds force $\mu\leq 23/7$, as $\epsilon<2$, which contradicts the condition that $B^3\geq 59$.

Now, let $C$ be non-degenerate with $q\geq 4$. Thus, $S_2$ is not abelian and as in the previous paragraph, Proposition \ref{prop:seshadriwidth} implies that $S_2$ is singular along $C$ and hence not normal. Moreover, $S_1$ is singular at the origin and normal, as an LC-locus. Hence,  $S_1\neq S_2$ and by $(\ref{eq:main5})$ and $(\ref{eq:main4})$, we get 
\[
t_1 \ \leq 6-\epsilon \textup{ and } t_{S_2}\ \leq \ 2\epsilon \ .
\]
On the other hand, by $(\ref{eq:main3})$ and Proposition~\ref{prop:nakamaye}, for any $\mu\geq t\geq 6-\epsilon$ we also get the inequalities
\[
\mult_{\overline{S}_1}(||B_t||)\ \geq t-4+(2\epsilon)/3  \ \textup{ and } \ \mult_{\overline{S}_2}(||B_t||)\ \geq t-2\epsilon \ .
\]
In this range $t-2\epsilon\geq 0$. As both $S_1$ and $S_2$ are singular at $0\in A$, Proposition~\ref{prop:nobodies} yields the inclusion
\[
\tilde{\Delta}_{Y_{\bullet}}(B) \ \cap \ \{t\}\times\RR^2 \ \subseteq \ \textup{convex hull}\{(t,0,0),(t,x_t,0),(t,0,x_t)\}  \ , \ \forall \ t\geq 6-\epsilon \ ,
\]
where $x_t= (24+8\epsilon-9t)/3$. Moreover, as $x_t$ has to be positive, then $\mu\leq (24+8\epsilon)/9$. 

For the upper bounds of the vertical slices of $\tilde{\Delta}_{Y_{\bullet}}(B)$, we use Theorem~\ref{thm:propertiesnobody} on $[0,\epsilon]\times\RR^2$  and Theorem~\ref{thm:abelian} on $[\epsilon,6-\epsilon]\times\RR^2$ ($q\geq 4$). Also, let $A_{\epsilon}=1$, when the region $[6-\epsilon,(24+8\epsilon)/9]\times\RR^2$ is non-trivial, i.e. $\epsilon\geq 30/17$,  and $A_{\epsilon}=0$ otherwise. Applying Fubini's theorem, these yield
\[
\frac{58}{6}  \leq \int_0^{\epsilon}\frac{t^2}{2} dt + \int_{\epsilon}^{\frac{4\epsilon}{3}} \frac{t^2-4(t-\epsilon)^2}{2} dt + \int_{\frac{4\epsilon}{3}}^{6-\epsilon}\frac{2\epsilon^2}{3} dt + A_{\epsilon}\cdot  \int_{6-\epsilon}^{\frac{24+8\epsilon}{9}}\frac{(24+8\epsilon-9t)^2}{18} dt  .
\] 
The last integral is at most $\frac{32}{243}$, attained at $\epsilon=2$. So, algebraic computations yield the inequality $32+972\epsilon^2-284\epsilon^3 > 2349$, which does not hold for $\epsilon\in[1.5,2]$, getting the desired contradiction.

\textbf{\textup{Case 3.3:}} $2\ \leq \ \epsilon(B;0) \ < \ 3$.

Here, the Seshadri exceptional curve $C$ (or surface) might be elliptic (abelian). So, we cannot rely on $S_2$, but the non-abelian surface $S_1$ will suffice. To start, as in Case~3.2 we have the inequalities
\[
t_1 \ \leq \ 6-\epsilon \textup{ and } \mu \ \leq \ 8-4\epsilon/3 \ .
\]
The first one is the same as $(\ref{eq:main4})$. The second one follows from $(\ref{eq:main3})$ and Proposition~\ref{prop:nobodies}.

Going back to the set $\tilde{\Delta}_{Y_{\bullet}}(B)$, we use Theorem~\ref{thm:propertiesnobody} on $[0,\epsilon]\times\RR^2$, and Proposition~\ref{prop:nobodies}, for the curve $C$, on $[\epsilon,6-\epsilon]$. On $[6-\epsilon,8-4\epsilon/3]\times\RR^2$ we apply again Proposition~\ref{prop:nobodies}, taking into account that $\mult_0(S_1)=2$ and $(\ref{eq:main3})$. By Fubini's theorem, all this data yields the inequalities
\[
\frac{59}{6} \ < \ \textup{vol}_{\RR^3}\Big(\tilde{\Delta}_{Y_{\bullet}}(B)\Big)\ \leq \ \int_0^{\epsilon}\frac{t^2}{2} dt \ + \ \int_{\epsilon}^{6-\epsilon} \frac{t^2-(t-\epsilon)^2}{2} dt   \ + \ \int_{6-\epsilon}^{8-\frac{4\epsilon}{3}} \frac{(8-\frac{4\epsilon}{3}-t)^2}{2} dt \ ,
\]
that imply $59 <(6-\epsilon)^3-(6-2\epsilon)^3 + (2-\epsilon/3)^3$. Using the derivative trick, the right hind-side is maximal at $\epsilon=2$, and its value there contradicts the inequality and we are done.
\end{proof}

\subsection{Globally generatedness}

We give a numerical criterion for an ample line bundle on an abelian $3$-fold to be base point-free. In particular, it proves Conjecture \ref{conj:main1} for $p=-1$ in dimension three.
\begin{theorem}\label{thm:globally}
	Let $(A,L)$ be a polarized abelian $3$-fold with $L^3>27$. The following are equivalent:
	\begin{enumerate}
		\item $L$ is globally generated.
		\item For any abelian subvariety $A'\subseteq A$ of dimension $d$ we have $(L^d\cdot A')>d^d$.
	\end{enumerate}
\end{theorem}
	\begin{example}{(\textbf{Line bundles with base points})}
Asymptotic Riemann-Roch on abelian three-folds yields $(L^3)=6h^0(A,L)$. So, Theorem \ref{thm:globally} doesn't cover when $(L^3)=6,12,18,24$. In the first three cases, if $L$ would be globally generated, it would define a morphism to a corresponding $\PP^r$ with $r=0,1,2$. So, this morphism would contract curves. And this is not possible as $L$ is ample. 
		
When $L^3=24$, a polarization of type $(1,1,4)$ can be either globally generated or not by \cite[Proposition 2 and Remark 3]{DHS94}, a phenomena better explained in terms of moduli of such pairs. Moreover, a polarization of type $(1,2,2)$ has always base points by \cite[Proposition 2.5]{NR95}.
	\end{example}
\begin{remark}
Due to the previous example, we don't believe Conjecture \ref{conj:main1} for $p=-1$ to be an equivalence in higher dimension. This behaviour does not happen in dimension one and two. Moreover, base-point freeness is easier to tackle than higher syzygies, as it is implied by the existence of divisors, which locally (not globally) have a zero-dimensional log-canonical center.
\end{remark}

\begin{proof}[Proof of Theorem \ref{thm:globally}]
$(1)\Longrightarrow (2)$ Let $A'\subseteq A$ be an abelian subvariety. The line bundle $L'=L|_{A'}$ is ample and base point free. So, when $A'$ is an elliptic curve, the properties of $L'$ force it to be of degree at least two. When $A'$ is an abelian surface and $L'^2\leq 4$, then $L'$ has at most two global sections. So, to be free, it would contract curves, and this is not possible as $L'$ is ample. 
	
$(2)\Longrightarrow (1)$ Using the ideas and notation from the proof of Theorem \ref{thm:main}, by asymptotic Riemann-Roch $L^3=6k$, for some $k\geq 5$. Thus $\mu(L;0)>3$ and for a general effective $\QQ$-divisor $\overline{D}\equiv B_t$ for any $t\in (3,\mu)$, then $\textup{lct}(D)<1$, where $D\equiv L$ is its pushforward. If $\textup{LC}(D)$ is zero-dimensional, then the ideas from the first part of the proof of \cite[Theorem 10.4.2]{PAG} imply $L$ is globally generated.

Let $\textup{LC}(D)$ be a smooth curve $C$. If $(L\cdot C)\geq 3$, by \cite[Proposition 3.2]{H97} we can find a new divisor with zero-dimensional LC-locus, and we are done. So, let $(L\cdot C)=1,2$. If $C$ is non-degenerate, as in Proposition \ref{prop:debarre}, \cite{D94} leads to a contradiction. If $S=C+C$ is abelian, Hodge index yields 
	\[
2\cdot (L^2\cdot S)\ \leq \ 	(L^2\cdot S)\cdot C^2 \ \leq \ (L\cdot C)^2 \ = \ 4 \ ,
	\]
	as $C$ is not elliptic, again leading to a contradiction to our initial asumptions.

So, we assume $C$ is elliptic. If $(L\cdot C)=1$, \cite[Lemma 1]{DH07} forces $(A,L)$ to be decomposable, and our initial conditions imply easily globally generatedness. When $(L\cdot C)=2$, first note that $L|_C$ is base point free. Furthermore, it it is not hard to see that generically the multiplier ideal of $\textup{lct}(D)\cdot D$ is equal to the ideal $\sI_{C|A}$. In particular, we have the following exact sequence
	\[
	0\ \rightarrow \ \sJ(A;\textup{lct}(D)\cdot D)\otimes L\ \rightarrow \ \sI_{C|A}\otimes L \ \rightarrow \ \mathscr{K}\otimes L \ \rightarrow \ 0  
	\]
	where $\mathscr{K}\otimes L$ has zero-dimensional support. As a consequence of this and Nadel vanishing we have $H^1(A, \sI_{F|A}\otimes L)=0$. Hence, the restriction map $H^0(A,L)\rightarrow H^0(C, L|_C)$ is surjective. But $L|_C$ is globally generated, so we can find a section of $L$ on $A$ that does not vanish at the origin. This implies $L$ is base-point free, as we can do the same for any point, as loing as we move $D$ around.
	
Finally, let the LC-center of $D$ be a surface $S\subseteq A$. If $S$ is abelian then Proposition \ref{prop:inductive} and our assumptions in the statement imply easily the statement. When $S$ is not abelian, by Lemma \ref{lem:nakayama} $\mult_0(S)\geq 2$ and as in the proof of Theorem \ref{thm:main}, we can assume further
\[
\mult_{\overline{S}}\big(||B_t||\big) \ \geq \ t/3 \ ,
\]
where $\overline{S}$ is the proper transform on the blow-up $\overline{A}$.

Now, translating this data to the convex set $\tilde{\Delta}_{Y_{\bullet}}(L)$, for a very generic flag $Y_\bullet$, and taking $t\rightarrow 3$, we get by Proposition \ref{prop:nobodies} the following inclusion
\[
\tilde{\Delta}_{Y_{\bullet}}(L) \ \bigcap \ \{3\}\times \RR^2 \ \subseteq \ \{(x,y)\ | \ 0\leq x,y, x+y\leq 1\} \ .
\]
This inclusion and convexity force $\tilde{\Delta}_{Y_{\bullet}}(L)$ to sit below a plane through  $(3,1,0)$ and $(3,0,1)$, which intersects the diagonal plane in a line through $(x,x,0)$ and $(x,0,x)$. By Theorem \ref{thm:propertiesnobody}  and \cite{N96}, we have $3>x\geq \epsilon(L,0)>1$ and we divide the proof in two cases, based on the magnitude of $x$.

Consider first $x\in [2,3)$. Our plane intersects at $(2x/(x-1),0,0)$ the horizontal line. So, $\tilde{\Delta}_{Y_{\bullet}}(L)$ is contained in the convex set generated by this point, $(0,0,0)$,$(x,x,0)$, and $(x,0,x)$. Using Fubini theorem, this inclusion gives rise to the following inequality
\[
30/6 \ \leq \ \textup{vol}_{\RR^3}\big(\tilde{\Delta}_{Y_{\bullet}}(L)\big) \ \leq \ x^3/(3x-3) \ ,
\]
which is not hard to see that doesn't hold for $x\in [2,3)$.

Consider $x\in(1,2)$. The same happens for $\epsilon(L,0)$. We assume $\epsilon(L;0)$ is defined by a curve, as the argument below still works when taking a limit. We show first that this curve is non-degenerate. If it is elliptic we get a contradiction, as $L$ is a line bundle. If it generates an abelian surface, our initial assumptions, Proposition \ref{prop:inductive} and \ref{prop:abelian surface} lead again to a contradiction. 

Consequently, using $\overline{S}$ and the Seshadri curve in Corollary \ref{cor:width}, we get $\mu(L,0)\leq 2+x$. Finally, the same exact ideas as in the previous case lead to the inequality
\[
\frac{30}{6} \ \leq \ \textup{vol}_{\RR^3}\big(\tilde{\Delta}_{Y_{\bullet}}(L)\big) \ \leq \ \int_0^{x}\frac{t^2}{2} dt \ + \ \int_{x}^{2+x} \frac{\big(\frac{x-1}{x-3}t-\frac{2x}{x-3}\big)^2}{2} dt   \  .
\]
Algebraic manipulation translates it into $0\leq x^5-15x^3-16x^2-200x-262$, which can be checked to not hold for $x\in (1,2)$. So, we finish the proof. 
	\end{proof}

\subsection{Singular divisors vs. abelian subvarieties}
In this subsection we present a proof of Corollary \ref{cor:main} on the existence of effective divisors, whose singularity locus is an abelian subvariety.
\begin{proof}[Proof of Corollary~\ref{cor:main}]
By the proof of Theorem~\ref{thm:main}, we can assume that for any $t\geq 3$ a general divisor $\overline{D}_t\equiv B_t$ has $\textup{LC}(D_t)$ to be either a non-elliptic curve $C_t$ or a non-abelian surface. Let $t_1$ be the maximum $t\in (3,\mu)$ with $\textup{LC}(D_t)$ being one-dimensional. So,  $\Bplus(B_t)$ for any $t>t_1$ contains a surface $S_t$, which is the LC-locus of a general $D_t$. By Lemma~\ref{lem:nakayama}  $S_t$ is singular at the origin.

Fix $t\in [3,t_1)$. Our first goal is to show that the bad case implies $(B\cdot C_t)\geq 5$. Let $C_t$ be degenerate and $S=C_t+C_t$ is an abelian surface. If $B-S$ is ample, by Proposition~\ref{prop:abelian surface} we are done. Otherwise, Hodge index on $S$, Proposition~\ref{prop:abelian surface} and $C_t^2\geq 2$, as $C_t$ is not elliptic,
\[
(B\cdot C_t)^2 \ \geq \ (B^2\cdot S)\cdot {C_t^2} \ \geq (2B^3)/3 \ > \ 25 \ .
\]
When $C_t$ is non-degenerate, then \cite{D94}, as in Proposition \ref{prop:debarre}, again implies $(B\cdot C_t)\geq 5$.

Now, inversion of adjunction techniques as in Step 1 from the proof of Theorem~\ref{thm:main} imply that either there is a divisor, whose LC-locus is the origin, or there is an upper bound of $(B\cdot C_t)$. With the above lower bound, and the fact that this holds for any $t<t_1$, we obtain the inequalities
\begin{equation}\label{eq:main7}
0 \ < \ \epsilon(B;0) \ \leq \ 15-4t_1 \ .
\end{equation}
In particular, $t_1< 3.75$ and we divide the proof in two parts based on the range of $t_1$. 

\textbf{\textup{Case 1:}} $3.5\leq t_1< 3.75$ ($\Rightarrow \ 0<\epsilon \leq 1$).

Suppose first that $\epsilon(B,0)$  is defined by a surface $S$. If it is abelian, then we would be done by Proposition~\ref{prop:abelian surface}.  Otherwise, Corollary~\ref{cor:width} forces $\mu(B;0)\leq 2$, contradicting $B^3>40$.

Thus, we can assume $\epsilon(B,0)$ is defined by a curve $C$. Let $C_3$ be a non elliptic curve in the LC-locus of a general divisor $D_3$. Thus, $S=C+C_3$ is a surface. Moreover, Proposition~\ref{prop:nakamaye} yields
\[
\mult_{\overline{C}_3}(||B_{t}||) \ \geq \ t-2, \textup{ and } \mult_{\overline{C}}(||B_{t}||) \ \geq \ t-1 \ ,
\]
for any $t\geq 3$. But then applying $(\ref{eq:main7})$ and Lemma~\ref{lem:uppermu}, this yields
\[
\mult_{\overline{S}}(||B_t||) \ \geq \ 4t_1+t-17 \ , \textup{ for any }\  3\leq t\leq t_1 \ .
\]
In particular, $\mult_{\overline{S}}(||B_4||)\geq 1$. So, if $S$ is abelian, Proposition~\ref{prop:abelian surface} yields $\mu(B;0)\leq 4$. Otherwise, $\mult_0(S)\geq 2$, by Lemma~\ref{lem:nakayama} and $S+C=X$. By Proposition~\ref{prop:nobodies}, again we get $\mu(B;0)\leq 4$.

Now, as usual we divide $\tilde{\Delta}_{Y_{\bullet}}(B)$ in regions. On $[0,15-4t_1] \times\RR^2$ we use Theorem~\ref{thm:propertiesnobody}, on $[15-4t_1,3]\times \RR^2$ Theorem \ref{thm:abelian} for the curve $C$,  and Proposition \ref{prop:decreasing} on $[3,t_1]\times\RR^2$ for surface $S$, and on $[t_1,4]\times\RR^2$ for $S_{t_1}$. As $\mult_0(S_1)=2$, then Fubini-s theorem yield the following inequalities
\[
\frac{40}{6} < \textup{vol}_{\RR^3}\big(\tilde{\Delta}_{Y_{\bullet}}(B)\big) \leq  \int_0^{15-4t_1}\frac{t^2}{2} dt  +  \int_{15-4t_1}^{3} \frac{t^2-(t-(15-4t_1))^2}{2} dt   +  \int_{3}^{t_1} \frac{9}{2} dt  +   \int_{t_1}^{4} \frac{(\frac{4t_1}{3}-t)^2}{2} dt \ .
\]
Going further, algebraic computation lead to a much simpler one
\[
40 \ < \  27+t_1^3/27\ + \ (t_1-3)\big(27-64(t_1-3)^2-64(t_1-3)^2/27\big)  \ .
\]
As $(t_1-3)$ is maximal at $t_1=3.5$ and the second term when $t_1=3.75$, yield a contradiction.

\textbf{\textup{Case 2:}} $3<t_1< 3.5$ ($\Rightarrow \ 1<\epsilon< 3$).

Here we use only the surface $S_1$. As in Case~3.2 in the proof of Theorem~\ref{thm:main}, we can prove
\[
\mu(B, 0) \ \leq \ 4t_1/3 \ .
\]
So, for the set $\tilde{\Delta}_{Y_{\bullet}}(B)$, we use Theorem~\ref{thm:propertiesnobody} on $[0,15-4t_1] \times\RR^2$. On $[15-4t_1,t_1]\times \RR^2$ Theorem \ref{thm:abelian} for the curve $C$, and Proposition~\ref{prop:nobodies} on $[t_1,\frac{4t_1}{3}]\times\RR^2$ for the surface $S_1$, which is singular at the origin. So, Fubini's theorem implies the following inequalities
\[
\frac{40}{6} \ < \ \textup{vol}_{\RR^3}\Big(\tilde{\Delta}_{Y_{\bullet}}(B)\Big)\ \leq \ \int_0^{15-4t_1}\frac{t^2}{2} dt \ + \ \int_{15-4t_1}^{t_1} \frac{t^2-(t-(15-4t_1))^2}{2} dt   \ + \ \int_{t_1}^{\frac{4t_1}{3}}  \frac{(\frac{4t_1}{3}-t)^2}{2} dt \ .
\]
Algebraic computations leads then to a much simpler one $40<t_1^3-(5t_1-15)^3+ t_1^3/27$. This doesn't hold for $t_1\in(3,3.5)$, so we finish the proof of the corollary.
\end{proof}

\end{document}